\documentclass[12pt, twoside]{article}
\usepackage{amsmath,amsthm,amssymb}
\usepackage{graphicx}
\usepackage{times}
\usepackage{enumerate}

\def\titlerunning#1{\gdef\titrun{#1}}
\makeatletter
\def\author#1{\gdef\autrun{\def\and{\unskip, }#1}\gdef\@author{#1}}
\def\address#1{{\def\and{\\\hspace*{18pt}}\renewcommand{\thefootnote}{}%
\footnote {#1}}%
\markboth{\autrun}{\titrun}}
\makeatother
\def\email#1{e-mail: #1}
\def\subjclass#1{{\renewcommand{\thefootnote}{}%
\footnote{\emph{Mathematics Subject Classification (2010):} #1}}}
\def\keywords#1{\par\medskip
\noindent\textbf{Keywords.} #1}

\newtheorem{thm}{Theorem}[section]
\newtheorem{cor}[thm]{Corollary}
\newtheorem{lem}[thm]{Lemma}

\newtheorem{proposition}[thm]{Proposition}

\theoremstyle{definition}
\newtheorem{defin}[thm]{Definition}

\newtheorem{notation}[thm]{Construction}

\numberwithin{equation}{section}

\begin{document}


\baselineskip=17pt


\titlerunning{Non-computable basins of attraction}

\title{An analytic system with a computable hyperbolic sink whose basin of attraction is non-computable}

\author{Daniel S.~Gra\c{c}a
\and 
Ning Zhong}

\date{}

\maketitle

\address{F1. Gra\c{c}a: CEDMES/FCT, Universidade do Algarve, Portugal
      \& SQIG - Instituto de Telecomunica\c{c}\~{o}es, Portugal; \email{dgraca@ualg.pt}
\and
F2. Zhong: DMS, University of Cincinnati, U.S.A.; \email{zhongn@uc.edu}}



\begin{abstract}
In many applications one is interested in finding the stability regions (basins of attraction) of some stationary states (attractors). In this paper we show that one cannot compute, in general, the basins of attraction of even very regular systems, namely analytic systems with hyperbolic asymptotically stable equilibrium points. To prove the main theorems, a new method for embedding a discrete-time system into a continuous-time system is developed.

\end{abstract}

\section{Summary of the paper}

In the study of dynamical systems, asymptotically stable equilibrium points and closed orbits are important since they represent stationary or repeatable behavior. In fact, to determine stability regions (basins or domains of attraction) of asymptotically stable equilibrium points is a fundamental problem in nonlinear systems theory with great importance in a number of applications such as in the fields of engineering (electric power system, chemical reactions), ecology, biology, economics, etc. In the late 1960's there was a surge of theoretical studies analyzing properties of such domains. In recent years much effort has been devoted to development of numerical methods for estimation of these domains, which has resulted in numerous numerical algorithms (see e.g.~\cite{Che04}, \cite{MBB10}). In contrast, relatively little theoretical work on computability of these domains exists.


It is known that these domains are non-computable in some instances. For example, by ``gluing'' different dynamics over different regions of the space, one may obtain $C^k$-systems (even $C^{\infty}$-systems) with domains of attraction which encode non-computable problems and which are thus non-computable \cite{Moo91}, \cite{Zho09}.

On the other hand, it is also known that for hyperbolic rational functions, there are (polynomial-time)
algorithms for computing basins of attraction and their complements
(Julia sets) with arbitrary precision \cite{Bea91}; in other words, basins of
attraction and Julia sets of hyperbolic rational functions are
(polynomial-time) computable.\nocite{GZB07}

So the question is, where does the boundary between computability and non-computability of basins of attraction lie? In particular, the question of computability remains open for (analytic) non-rational systems. Analyticity imposes a strong degree of regularity, which is higher than that of $C^{\infty}$ continuity, since the local behavior of an analytic function
determines how it behaves globally; thus no ``gluing'' of different dynamics in different regions is allowed. Another common requirement of regularity found in dynamical systems theory is hyperbolicity (see e.g.~\cite{Per01}). In short, hyperbolicity requires that near attractors (stationary states), the flow must converge to these attractors at a (at least) uniform rate to avoid pathological behavior due to a convergence which is ``too slow''. In particular, if we consider the simplest type of attractors, i.e.~asymptotically stable equilibrium points, hyperbolicity implies that near an (hyperbolic) asymptotically stable equilibrium point the flow must converge to this point at a rate which is equal to or greater than $e^{-\lambda t}$, for some contraction rate $\lambda>0$.


In this paper we show that: \\

\textbf{Main Theorem.} There is an analytic and computable dynamical
system with a computable hyperbolic sink $s$ such that the basin of attraction
of $s$ is not computable (technically more precise results are given in Theorem
\ref{Th:Mainresultprecise} and Corollary \ref{Cor:Main} below). \\

Thus our result implies that no algorithmic characterization
exists, in general, for a given basin of attraction even if an high degree of regularity (analicity + hyperbolicity) is imposed on the system and even if only the simplest type of attractors (equilibrium points) is considered.

In the case of discrete-time systems, we prove the
result by encoding a well-known non-decidable problem into the basin
of attraction of $s$. In the case of continuous-time systems, we
prove the result by embedding a discrete-time system with a
non-computable basin of attraction into a continuous-time system.
The standard suspension method (see Smale \cite{Sma67}, Arnold and Avez \cite{AA68})
for embedding  a discrete-time system into a continuous-time
system is not sufficient for our case; we instead develop a new method.

The structure of the paper is as follows. In Section \ref{Sec:Intro} we
briefly mention related work and then discuss basic notions from dynamical
systems and from computability with real numbers. We also delve into our
results in a more precise way (see Theorems \ref{Th:Mainresultprecise} and
Corollary \ref{Cor:Main}). Then in Section \ref{Sec:Mainresults} we present
two formulations of Theorems \ref{Th:Mainresultprecise}: one for the
discrete-time case and another for the continuous-time case. Each case will be
proved separately in Sections 4 and 5. Some proofs become quite technical, so
we provide a brief guide in Section \ref{Sec:Results roadmap}.

\section{Introduction}\label{Sec:Intro}

\subsection{Related work}

Analytic dynamical systems lie between polynomial and $C^{\infty}$ systems.
As noted in the summary, basins of attraction of dynamical systems generated by
hyperbolic polynomials are polynomial-time computable. On the other hand, it
is shown in \cite{Zho09} that for $C^{\infty}$ dynamical systems, basins of
attraction of hyperbolic sinks can be non-computable. Thus the case of analytic
systems is the cutoff point in terms of computability of basins.

By the main theorem of this paper, there is a dynamical system
generated by an analytic function $f$ with a hyperbolic sink $s$
such that the basin of attraction of $s$ is non-computable, even
though for the corresponding systems generated by any initial finite
segment of the Taylor series of $f$, the basins of $s$ are
polynomial-time computable.  This seems a bit surprising since it is
well known that, in contrast to the $C^{\infty}$ case, analytic
functions enjoy pervasive computability; for example, the sequence
$\{ f^{(n)}\}$ of derivatives is computable if the computable
function $f$ is analytic but may fail to be so if $f$ is only
$C^{\infty}$. On the other hand, the non-computability in the
analytic case cannot be proved as for the $C^{\infty}$ case, because
the proof for the $C^{\infty}$ case makes crucial use of the fact
that a non-constant $C^{\infty}$ function can take a constant value
on a non-empty open subset; such a function cannot be analytic.  The
idea underlying the proof of the $C^{\infty}$ case is that, starting
with a non-computable open set, one constructs two sequences of
$C^{\infty}$ functions such that one contracts on the open set and
the other expands on the complement of the set. Then one glues the
two sequences together to produce a $C^{\infty}$ system such that
the non-computable open set gives rise to the desired non-computable
basin \cite{Zho09}. This construction will not work for analytic
systems, for the local behavior of an analytic function determines
how it behaves globally. A completely different construction is
developed in this paper in order to show that basins of attraction
can be non-computable, even in the case of analytic systems.

An interesting related result can be found in \cite{BY06}, 
which studies the computability of Julia sets
defined by quadratic maps. The authors show that Julia sets are, in
general, not computable, although hyperbolic Julia sets are computable, in a very
efficient manner (in polynomial time) \cite{Bra05}, \cite{Ret05}. This somehow
suggests that hyperbolicity makes problems computationally simpler. Since hyperbolicity is not
enough to guarantee computability in our case, it seems that
computing basins of attraction may even be more challenging than
computing Julia sets.
Note that hyperbolicity also makes the problem of
computing basins of attraction computationally simpler \cite{Zho09} (there are
several degrees of non-computability: it is possible to say that a
non-computable problem is \textquotedblleft harder\textquotedblright\ than
another problem -- cf.\ \cite{Odi89}, \cite{Odi99}), but not enough so that
computability is achieved.

\subsection{Dynamical systems and hyperbolicity}

We recall that there are two broad classes of dynamical systems: discrete-time
and continuous-time (for a general definition of dynamical systems,
encompassing both cases, see \cite{HS74}). A discrete-time dynamical system is
defined by the iteration of a map $f:\mathbb{R}^{n}\rightarrow\mathbb{R}^{n}$,
while a continuous-time system is defined by an ordinary differential equation
(ODE) $x^{\prime}=f(x)$. Common to both classes of systems is the notion of
trajectory: in the discrete-time case, a \emph{trajectory} starting at the
point $x_{0}$ is the sequence%
\[
x_{0},f(x_{0}),f(f(x_{0})),\ldots,f^{[k]}(x_0),\ldots
\]
where $f^{[k]}$ denotes the kth iterate of f, while in the
continuous time case it is the solution to the following
initial-value problem%
\[
\left\{
\begin{array}
[c]{l}%
x^{\prime}=f(x)\\
x(0)=x_{0}%
\end{array}
\right.
\]
Trajectories $x(t)$ may converge to some \emph{attractor}.
Attractors are invariant sets in the sense that if a trajectory
reaches an attractor, it stays there.
Given an attractor $A$, its \emph{basin of attraction} is the set%
\[
\{x\in\mathbb{R}^{n}|\text{ the trajectory starting at }x\text{
converges to }A \text{ as } t\rightarrow\infty \}
\]
Attractors come in different varieties: they can be points, periodic orbits,
strange attractors, etc. In this paper we focus on the simplest type of
attractors: single points (also called \textit{fixed points}). If there is a
neighborhood around a fixed point $s$ which is contained in the basin of
attraction of $s$ (i.e.\ every trajectory starting in this neighborhood will
converge to $s$), then $s$ is called a sink.

Among smooth dynamical systems, hyperbolic systems play a central role.
In the case of a sink $s$, hyperbolicity amounts to
requiring a uniform rate of convergence at which every trajectory starting in a neighborhood of $s$
converges to $s$. More details can be found in \cite{HS74}.

\subsection{Introduction to computability over real numbers\label{Sec:Computability}}

\textbf{Classical Computability.} At the heart of computability
theory lies the notion of \emph{algorithm}. Roughly speaking, a
problem is \emph{computable} if there is an algorithm that solves
it. For example, the problem of finding the greatest common divisor
of two positive integers is computable since this problem can be
solved using Euclid's algorithm. On the other hand, there are many
non-computable problems. For example, Hilbert's 10th problem is not
solvable by any algorithm. 
But how can we show that no algorithm solves this problem? To answer
such questions, the notion of algorithm has to be formalized.
The formal notion of algorithm makes use of \emph{Turing machines}
(see e.g.\ \cite{Sip05}, \cite{Odi89} for more details), which were
introduced in 1936 by Alan Turing \cite{Tur36}. The very simple
behavior of a Turing machine (TM for short), which mimics human
beings executing an algorithm, and a series of equivalence results
between different models of computation led to the following
conclusion (\textit{Church-Turing thesis}): problems solvable by
algorithms (ordinary computers) are exactly those solvable by Turing
machines. Notice that the Church-Turing thesis cannot be proved --
it simply formalizes the notion of algorithm -- but\ it is
unanimously accepted by the scientific community.

Classical computability is carried out over discrete structures.
Formal definitions usually involve strings of symbols (e.g.\ binary
words), but they can  be stated equivalently using natural numbers
as follows. Let $\mathbb{N}$, $\mathbb{Z}$, $\mathbb{Q}$, and
$\mathbb{R}$ denote the set of natural numbers (including 0),
integers, rational numbers, and real numbers, respectively.

\begin{defin}
A function $f:\mathbb{N}^{k}\rightarrow\mathbb{N}$ is computable if
there is a Turing machine that on input
$(x_{1},\ldots,x_{k})\in\mathbb{N}^{k}$ outputs the value
$f(x_{1},\ldots,x_{k})$.
\end{defin}

In many applications, given an input $(x_{1},\ldots,x_{k})\in\mathbb{N}^{k}$,
we are only interested in a \textquotedblleft yes/no\textquotedblright%
\ output. This can be done, without loss of generality, by assuming that
$f(x_{1},\ldots,x_{k})=1$ is a \textquotedblleft yes\textquotedblright\ answer
and $f(x_{1},\ldots,x_{k})=0$ is a \textquotedblleft no\textquotedblright%
\ answer.\smallskip

\textbf{Computing with real numbers.} In order to study
computability problems over $\mathbb{R}^{n}$, one has to generalize
the previous notions. Turing, in his seminal paper \cite{Tur36},
already provided an approach: code each real number as its decimal
expansion and then carry out computations over this symbolic
representation. Although simple, this approach does not work as
noted by Turing himself in a later paper \cite{Tur37}, since trivial
functions like $x\mapsto3x$ would not be computable in this setting.
This is because the decimal representation does not preserve the
topology of the real line: $0.9999\ldots$ and $1.0000\ldots$ are far
from each other if considered as strings of symbols, yet they
represent the same real number. Details can be found in
\cite{Wei00}, \cite{BHW08}.

Nevertheless, the previous idea can still be used if we use a representation
of real numbers that preserves the topology of the real line. Here lies the
foundations for computable analysis, which draws on both computability theory
and topology, and allows one to compute over topological spaces in addition to
$\mathbb{N}$. In particular, for $\mathbb{R}$, among many different equivalent
possibilities, we use the following formulation: the function $\phi
:\mathbb{N\rightarrow Q}$ is an \emph{oracle} (also called $\rho$-name) for a
real number $x$ if $\left\vert x-\phi(n)\right\vert \leq2^{-n}$ for all
$n\in\mathbb{N}$. In other words, $\phi$ provides rationals which approximate
$x$ within any desired precision. We can then define \emph{oracle Turing
machines} by giving the TM access to an oracle: at any step of the computation
the TM can query the value $\phi(n)$ for any $n$. In this way one introduces
computability with real numbers.

\begin{defin}
A number $x\in\mathbb{R}$ is \emph{computable} if there is a TM which computes
an oracle $\phi$ for $x$: on input $n$, the machine outputs $\phi(n)$.
\end{defin}

Intuitively, a real number $x$ is computable if there exists a TM
which can compute a rational approximation of $x$ to any desired
precision. All familiar real numbers (rational numbers, algebraic
numbers, $e$, $\pi$, etc.) are computable. Notice that there are
only countably many Turing machines. Thus there are only countably
many computable real numbers.

What about computable functions? From the above considerations, one cannot
expect that a computable function $f:\mathbb{R\rightarrow R}$ generates only
computable real numbers. What we can expect is that there is an oracle TM such
that, given an oracle coding the input, the TM outputs $f(x)$.

\begin{defin}
A function $f:\mathbb{R\rightarrow R}$ is \emph{computable} if there
is an oracle TM $M$ such that if $\phi$ is an oracle for $x$,
then $M^\phi$ computes an oracle for
$f(x)$: on input $n$,  $M^{\phi}$ outputs a rational $q$ such that
$\left\vert q-f(x)\right\vert <2^{-n}$.
\end{defin}

Thus a function is computable if, given an oracle for the input, one
can compute the output to any desired precision. Elementary
functions from analysis are computable functions. These notions can
be extended in an obvious way to $\mathbb{R}^{n}$.

Finally, because we will study the computability of basins of
attraction of hyperbolic sinks (which are open sets of
$\mathbb{R}^{n}$), we need to define computable open subsets of
$\mathbb{R}^{n}$. This can be done by encoding the set into an
oracle (using the bases which generate the upper and lower Fell
topology) and showing that there exists a TM which computes this
oracle (see \cite{BHW08} for more details). However, for practical
reasons, we use instead the following equivalent definition.

\begin{defin}
An open set $O\subseteq\mathbb{R}^{n}$ is computable if the distance
function $d_{\mathbb{R}^{n}\backslash
O}:\mathbb{R}^{n}\rightarrow\mathbb{R}$ is computable, where
\[
d_{\mathbb{R}^{n}\backslash O}(x)=\inf_{y\in\mathbb{R}^{n}\backslash
O}d(x,y).
\]

\end{defin}

Intuitively an open set $O$ is computable if $O$ can be generated by
a computer with arbitrary precision. In this paper we'll show the
following result.

\begin{thm}
[main result]\label{Th:Mainresultprecise}There exists a dynamical
system defined by an analytic and computable function $f$, which
admits a computable hyperbolic sink $s$, such that the basin of
attraction of $s$ is not computable.
\end{thm}

We'll prove two versions of this result: one for discrete-time
systems and the other for continuous-time systems (Theorems
\ref{Th:Main_discrete} and \ref{Th:Main_continuous} of Section
\ref{Sec:Mainresults}).

Notice that computability does not need to be restricted to
$\mathbb{N}$ or $\mathbb{R}$. Given a topological space $X$ and a
suitable coding of elements of $X$ in the form of oracles (which
depends on the topology of $X$), one can define computable elements
of $X$ similarly to computable real numbers. Using the same idea,
one can futher define computable functions $f:X\rightarrow Y$. An
important result is the following (cf.\ \cite[Corollary
4.19]{BHW08}).

\begin{proposition}
\label{Prop:Computable}If $x\in X$ and $f:X\rightarrow Y$ are
computable, then $f(x)$ is computable.
\end{proposition}

In particular, if $x$ is computable, but $f(x)$ is not, then $f$
cannot be computable. Thus, when $X$ is taken to be a set of real
functions, 
Theorem \ref{Th:Mainresultprecise} implies the following corollary.

\begin{cor}
\label{Cor:Main}There is no algorithm which on input $(f,s)$, where $f$ is a
real analytic function and $s$ is a hyperbolic sink of the dynamical system
defined by $f$, computes the basin of attraction of $s$.
\end{cor}

In fact, Theorem \ref{Th:Mainresultprecise} is stronger than Corollary \ref{Cor:Main}.
Corollary \ref{Cor:Main} ensures that no \textit{single} algorithm can be
used to compute the basin of attraction for \textit{any} input $(f,s)$. This
is known as \textit{uniform computability}. A weaker version is
\textit{non-uniform computability}. In the latter case, we require again an
input $(f,s)$, with the difference being that different algorithms can be used
for different inputs. But Proposition \ref{Prop:Computable} also holds for
non-uniform computability and thus a non-uniform version of Corollary
\ref{Cor:Main} also holds.

It should be noted that computational problems about the long-term
behavior of dynamical systems have been discussed in \cite{BCSS98}.
Their notion of computability is, however, quite different from the
one used here. They allow the use of infinite precision calculations
but require that exact sets are generated. Under this model simple
sets like the epigraph of the exponential
$E=\{(x,y)\in\mathbb{R}^{2}|y\geq e^{x}\}$ are not computable
\cite{Bra00}. Thus in this model results do not correspond to
computing practice \cite{BC06}.

\section{Main results}\label{Sec:Mainresults}


The main theorem, Theorem \ref{Th:Mainresultprecise}, has the
following two versions:

\begin{thm}
[discrete-time case]\label{Th:Main_discrete}There is an analytic and
computable function $f:\mathbb{R}^{3}\rightarrow\mathbb{R}^{3}$ which defines
a discrete-time dynamical system with the following properties:

\begin{enumerate}
\item It has a computable hyperbolic sink $s$;

\item The basin of attraction of $s$ is not computable.
\end{enumerate}
\end{thm}

\begin{thm}
[continuous-time case]\label{Th:Main_continuous}There is an analytic
and computable function $g:\mathbb{R}^{6}\times (-1,+\infty)\times
(-1,+\infty)\rightarrow\mathbb{R}^{8}$ which defines a
continuous-time dynamical system via an ODE $y^{\prime}=g(y)$ with
the following properties:

\begin{enumerate}
\item It has a computable hyperbolic sink $s$;

\item The basin of attraction of $s$ is not computable.
\end{enumerate}
\end{thm}

We will prove the two theorems constructively in the sense that the
functions $f$ and $g$ are explicitly constructed.


\subsection{Road map to results\label{Sec:Results roadmap}}

The proofs will at some point become quite involved. For this reason
we provide a brief guide which may help the reader to keep track of
results.

In Section \ref{Sec:DiscreteTime} we prove Theorem \ref{Th:Main_discrete}. The
idea underlying the proof is to encode a non-computable problem into the basin
of attraction of $s$. Thus, if one can compute the basin of attraction of
$s$, then one can algorithmically solve a non-computable problem, a contradiction.

The problem which is encoded is the famous Halting problem. We encode an input
$x_{0}$ of the Turing machine into a triple in $\mathbb{N}^{3}$, the Turing
machine itself into a map $f:\mathbb{R}^{3}\rightarrow\mathbb{R}^{3}$, and the
halting configuration into a hyperbolic sink $s\in\mathbb{N}^{3}$ (assuming
w.l.o.g.\ that the halting configuration is unique). Thus the problem of determining whether a
trajectory starting at $x_{0}$ will converge to $s$ is equivalent to the
Halting problem and is therefore not computable.

Although conceptually simple, this idea is not so easy to implement
in practice. There are several issues to be dealt with. To begin
with, the value of $f(x)$ depends upon only a finite amount of
information carried by $x$. Thus one has to extract this information
first (using the tools of Section \ref{Sec:Preliminaries}; the exact
procedure is explained in the beginning of Section
\ref{Sec:3.NextAction}) and then compute $f(x)$ based on this finite
amount of information. Here techniques of interpolation are employed
(Section \ref{Sec:3.NextAction}).

Moreover, for technical reasons, $f$ needs to be a contraction near integers,
that is,
\begin{equation}
f(B(\alpha,\varepsilon))\subseteq B(f(\alpha),\varepsilon) \label{Eq:GoodProperty}%
\end{equation}
for some fixed $\varepsilon>0$, where $\alpha\in\mathbb{N}^3$ and
$B(\alpha,\varepsilon)=\{x\in\mathbb{R}^{3}|\left\Vert x-\alpha\right\Vert <\varepsilon\}$.
Thus we need to further study how error propagates through $f$ and
fine-tune $f$ whenever necessary (using functions defined in Section
\ref{Sec:Preliminaries}) so that $f$ satisfies
(\ref{Eq:GoodProperty}).
The details are given in Sections \ref{Sec:3.NextAction} and
\ref{Proof:Simulation}.

The continuous-time case is proved differently. The key idea
underlying the proof is to embed the previous discrete-time
dynamical system into a continuous-time system such that $s$ is
still a hyperbolic sink and its basin of attraction remains unchanged
through the embedding. Thus the basin of attraction of $s$ remains
non-computable in the continuous-time system. The challenge in this
approach is to obtain such an embedding (see Sections
\ref{Sec:Branickyinfty}, \ref{Sec:Branickyanalytic} and
\ref{Sec:Hyper} for the construction of an appropriate embedding).
Assuming the existence of such an embedding, Theorem
\ref{Th:Main_continuous} is proved in Section \ref{Sec:final}.
The remainder of this road map deals with Sections
\ref{Sec:Branickyinfty}, \ref{Sec:Branickyanalytic},
\ref{Sec:Hyper}, \ref{Sec:Simulate_TM}, and \ref{Sec:Hyperbolic}. The reader may prefer to
skip these sections in a first reading.

There is a standard technique for embedding discrete-time systems
into continuous-time systems, called \emph{suspension} \cite{AA68},
\cite{Sma67}. This procedure relies on equivalence relations and is
ill suited for our purposes. Indeed, if we pick a discrete-time
dynamical system defined by an analytic and computable function
$f:\mathbb{R}^{n}\rightarrow\mathbb{R}^{n}$ and $f$ admits a
hyperbolic sink $s$, then using the suspension method would yield a
continuous-time system $x^{\prime}=g(x)$ with one or more or all of
the following undesirable features:\smallskip

(i) $g$ may be non-analytic (discontinuities in $g^{(k)}$, the $k$th
order derivative of $g$, may occur);

(ii) $g$ may be non-computable (the procedure is not constructive);

(iii) $s$ may be part of an attractor, but this attractor may not be a sink
(it may be, for example, a cycle);

(iv) even if $s$ is a fixed point, there is no guarantee that it is hyperbolic
(in the continuous-time system).\smallskip

In this paper, we develop a technique for the embedding that is
different from the suspension method.
Our idea is based on the techniques developed in \cite{Bra95},
\cite{CMC00}, \cite{CM01}, \cite{Cam02b}. We observe that if a
discrete-time system defined by
$f:\mathbb{R}^{n}\rightarrow\mathbb{R}^{n}$ is embedded into a
continuous-time system $y^{\prime}=g(y)$, then the trajectory
$y(t)$, starting at $x_{0}$, must arrive at the states $f(x_{0})$,
$f(f(x_{0}))$, and
so on at times $t=1,2,\ldots$; that is, $y(0)=x_{0},y(1%
)=f(x_{0}),y(2)=f(f(x_{0})),\ldots$. This requirement of
$y(k)=f^{[k]}(x_{0})$ for all $k\in\mathbb{N}$ is what makes the
embedding difficult - there are infinitely many conditions,
$y(k)=f^{[k]}(x_0)$, to be met.
The idea used for tackling the problem is to enlist a second
companion system and use, alternatively, one system to guide the
trajectory of the other moving forward, satisfying one requirement
at a time. Let us look into the details a bit more. To embed the
discrete-time dynamical system defined by
$f:\mathbb{R}^{n}\rightarrow\mathbb{R}^{n}$ into a continuous-time
system, we design a continuous-time system $y^{\prime}=g(y)$ that is
defined on $\mathbb{R}^{2n}$. The system has two components $y_{A}$
and $y_{B}$, both of dimension $n$. The components $y_{A}$ and
$y_{B}$\ both start at value $x_{0}$ for $t=0$. In the time interval
$[0,1/2]$, $y_{B}$ is unchanged (serving as a memory) with the
constant value $x_{0}$ and the values of $y_{A}$ are updated so that
$y_{A}=f(y_{B})=f(x_{0})$ at $t=1/2$. In the next half-unit time
interval $[1/2,1]$, $y_{A}$ serves as a memory taking the constant
value $f(x_{0})$, while the values of $y_{B}$ are updated so that
$y_{B}=y_{A}$ at $t=1$, i.e.\ $y_{B}=f(x_{0})$. Thus at $t=1$, one
has $y_{A}(1)=y_{B}(1)=f(x_{0})$. The procedure is then repeated in
the subsequent time intervals $[1,2],[2,3],\ldots$, which leads to
the following desired property: %
\[
y_{A}(2)=y_{B}(2)=f(f(x_{0})),\text{ \ \ }y_{A}(3)=y_{B}(3)=f(f(f(x_{0}%
))),\text{ \ldots}.%
\]
The procedure is graphically depicted in Fig.\ \ref{fig:branicky}.%
\begin{figure}[ptb]
\begin{center}
\includegraphics[width=0.6\textwidth]{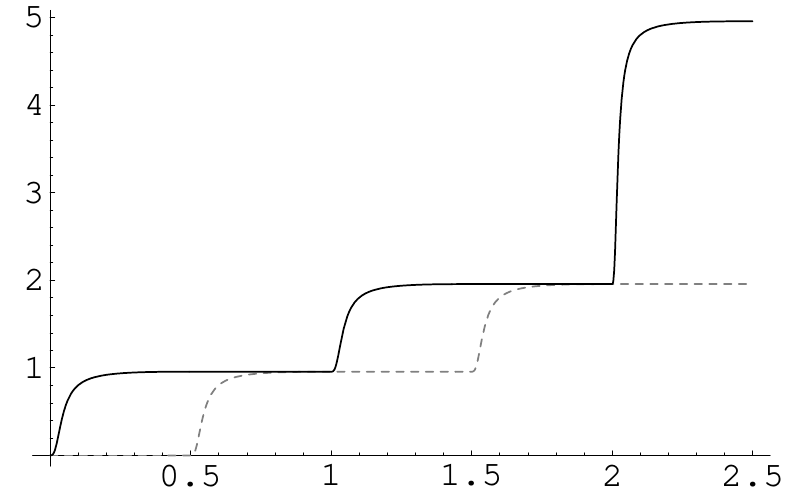}
\end{center}
\caption
{Simulation of the iteration of the map $f(n)=n^2+1$ via an ordinary differential equation, where $n$ is initially $0$.
The solid line
represents the variable $y_A$ and the dashed line represents $y_B$.}
\label{fig:branicky}
\end{figure}%

The subtlety is to always have at each moment a component which acts as a
\textquotedblleft memory\textquotedblright\ of the last value computed (see
Section \ref{Sec:Branickyinfty} for details). Using this technique, one can
embed the discrete-time dynamical system defined by $f:\mathbb{R}%
^{3}\rightarrow\mathbb{R}^{3}$ into a continuous-time system $y^{\prime}%
=g(y)$, where $g:\mathbb{R}^{6}\rightarrow\mathbb{R}^{6}$. Moreover,
if $f$ is computable, so is $g$. Thus the problem (ii) above is
solved satisfactorily. However, the remaining problems (i), (iii)
and (iv) have yet to be addressed. In particular, because the
components of $g$ vary in some open time intervals but take constant
values in others, $g$ cannot be analytic.

In this paper, we refine the previous embedding so that it becomes
analytic (if we start with an analytic function $f$); the problem
(i) is then subsequently solved. We believe that this cannot be done
in general, because the propagation of error due to a slight
perturbation on $x_{0}$ (thus on $f(x_{0})$ as well) may alter the
behavior of $y$ in the long run, resulting in $y_{A}(k)$ being far
different from $f^{[k]}(x_{0})$ for large $k$'s. Fortunately, in our
case, the function $f$ satisfies the inclusion
(\ref{Eq:GoodProperty}) near all integers, which in turn implies
that a sufficiently good
approximation to $x_{0}$ ($f(x_{0})$, respectively) can be used to
generate a trajectory that is a good enough approximation to the
entire trajectory defined by $f$ and starting at $x_{0}$. Using the
crucial inclusion (\ref{Eq:GoodProperty}), one can allow $y_{A}\
$and $y_{B}$ to be almost constant (but analytic) instead of
requiring them to be constant in some time intervals. Of course,
we'll then have $y_{A}(1)\simeq y_{B}(1)\simeq f(x_{0})$, but since
the propagation of error through the trajectory is well contained
due to (\ref{Eq:GoodProperty}), we will obtain
\begin{equation}
y_{A}(2)\simeq y_{B}(2)\simeq f(f(x_{0})),\text{ \ \ }y_{A}(3)\simeq
y_{B}(3)\simeq f(f(f(x_{0}))),\text{ \ldots} \label{Eq:approx}%
\end{equation}
Moreover the new system $y^{\prime}=g(y)$ will still be computable and
analytic (the construction is presented in Section \ref{Sec:Branickyanalytic}%
). At this point, problems (i) and (ii) are taken care of.

It remains to address problems (iii) and (iv). Notice that if
$x_{0}$ belongs to the basin of attraction of the discrete-time
dynamical system, then the previous embedding only ensures that the
(continuous-time) trajectory starting at $x_{0}$ will wander around
$s$ due to (\ref{Eq:approx}). It does not guarantee though that $s$
is a sink nor does it guarantee hyperbolicity. Hyperbolicity of $s$
(assuming that $s$ is indeed a sink) can be shown by demonstrating
that the Jacobian of the system at $s$ admits only eigenvalues with
negative real part (Section \ref{Sec:Hyperbolic}), thus solving
problem (iv). The last problem, problem (iii), is solved by changing
the dynamics of $y^{\prime}=g(y)$ so that all points in $B(s,1/4)$
will be converging to $s$ and, during the changing process, problems
(i), (ii), and (iv) are kept under control and do not resurface.
This guarantees that $s$ is a hyperbolic sink and therefore we have
our desired embedding. The latter result is proved in Section
\ref{Sec:Hyper}. The proof uses various techniques and spans several
pages. The reader may prefer to skip it on a first reading. Finally, Section \ref{Sec:Simulate_TM}
shows that the system which satisfies (i)-(iv) still simulates a Turing machine, which guarantees
non-computability of the basin of attraction for the hyperbolic sink.

\section{The discrete-time case}\label{Sec:DiscreteTime}

Recall that Turing machines emulate computer programs. One can always stop the
execution of a program and record the value of the variables and the line
where the program halted. This recorded information is called a
\emph{configuration} in the context of TMs and will be used below. It has all
the information we need to continue the computation, if we wish.

Each time we start a new computation, it starts on some special class of
configurations -- the \emph{initial configurations}. The computation goes on
until one eventually reaches a \emph{halting configuration} where the TM stops
its execution (it \textit{halts} -- which may never happen, i.e.\ the TM may
run forever). Let $M$ be a universal Turing machine (the precise definition of
a universal Turing machine is not important here, but can be found in
\cite{Sip05}). The Halting problem can be stated as follows: \textquotedblleft
Given some initial configuration $x_{0}$ of $M$, will the computation reach
some halting configuration?\textquotedblright\ This problem is well known to
be non-computable (cf.\ \cite{Sip05}). Without loss of generality, we may
assume that a TM has just one halting configuration (e.g.\ just before ending,
set all variables to 0 and go to some special line with a command
\texttt{break}; thus the final configuration is unique).

Configurations can be encoded as points of $\mathbb{N}^{3}$ (see Section
\ref{Sec:Configurations}). Since essentially what a TM does is to update one
configuration to the next configuration and so on, until the TM halts,
associated to each TM $M$ there is a transition function $f_{M}:\mathbb{N}%
^{3}\rightarrow\mathbb{N}^{3}$ which describes how $M$ behaves. To prove the main theorem
in the discrete-time setting, we need to extend $f_{M}$ to an analytic function $f:\mathbb{R}^3\to\mathbb{R}^3$
which is non-expanding around the inputs of the transition function $f_{M}$,
i.e., $f$ maps points near a configuration $x\in\mathbb{N}^{3}$
to points near $f(x)$.

The simulation of Turing machines with analytic maps was first studied in
\cite{KM99}. A different method was used in \cite{GCB08} to simulate Turing machines with analytic maps which has the advantage of presenting some robustness to perturbations. This robustness property can then be used to simulate of Turing machines in continuous time through the use of an ODE \cite{GCB08}. The results shown there guarantee that points near a
configuration $x$ will be mapped into some bounded vicinity of $f(x)$, which,
of course, does not ensure non-expansiveness. More concretely, it was proven
in \cite{GCB08} that for each Turing machine $M$, there exists an analytic map
$f:\mathbb{R}^{3}\rightarrow\mathbb{R}^{3}$ which simulates it (meaning the
restriction of $f$ to $\mathbb{N}^{3}$ is the transition function of $M$) with
the following property:
\begin{equation}
\left\Vert y-x\right\Vert \leq1/4\text{ \ \ }\Rightarrow\text{
\ \ }\left\Vert f(y)-f(x)\right\Vert \leq1/4\label{Eq:old_map_simulation}
\end{equation}
where $x\in\mathbb{N}^{3}$ codes a configuration. In this paper, the analytic
map $f$ needs to be non-expanding in order to have the halting configuration
as a sink; thus $f$ must satisfy the following stronger property: for every $0\leq\varepsilon\leq1/4$, the following holds true:
\[
\left\Vert y-x\right\Vert \leq\varepsilon\text{ \ \ }\Rightarrow\text{
\ \ }\left\Vert f(y)-f(x)\right\Vert \leq\varepsilon.
\]
To achieve this purpose, we adapt the results from \cite{GCB08}. In the
remainder of this paper we take%
\[
\left\Vert (x_{1},\ldots,x_{k})\right\Vert =\left\Vert (x_{1},\ldots
,x_{k})\right\Vert _{\infty}=\max_{1\leq i\leq k}\left\vert x_{i}\right\vert.
\]
The following result will be proved in Section \ref{Proof:Simulation}.

\begin{thm}
\label{Th:simulation}Let $M$ be a Turing machine, and let $f_{M}:\mathbb{N}%
^{3}\rightarrow\mathbb{N}^{3}$ be the transition function of $M$.
Then $f_{M}$ admits an analytic and computable extension
$f:\mathbb{R}^{3}\rightarrow \mathbb{R}^{3}$ with the following
property: there exists a constant
$\lambda\in (0, 1)$ with the property that for any $0<\varepsilon\leq 1/4$, if $x\in\mathbb{N}^{3}$
is a configuration of $M$, then for any $y\in\mathbb{R}^3$,%
\begin{equation}
\left\Vert x-y\right\Vert \leq\varepsilon\text{ \ \ \
}\Longrightarrow \text{ \ \ \ }\left\Vert f(x)-f(y)\right\Vert
\leq\lambda\cdot\varepsilon.
\label{Eq:Th_main}%
\end{equation}
\end{thm}

In the previous proposition we have assumed that if $x$ is a halting
configuration, then $f_{M}(x)=x$, i.e.\ $x$ is a fixed point of $f$.

\begin{lem}
\label{Lemma:hyperbolic}Let
$f:\mathbb{R}^{n}\rightarrow\mathbb{R}^{n}$ be a map with a fixed
point $x_{0}$, and let $B(x_{0}, r)$ be a neighborhood of $x_{0}$
with $r>0$. If there is a constant $\lambda\in(0,1)$ such that for
all $x\in B(x_{0}, r)$,%
\[
\left\Vert f(x)-f(x_{0})\right\Vert \leq\lambda\left\Vert x-x_{0}\right\Vert
\]
then $x_{0}$ is a hyperbolic sink of $f$. In particular, no
eigenvalue of $Df(x_{0})$ has absolute value$\ $larger that
$\lambda$.
\end{lem}

\begin{proof}
Since $f$ is analytic, we know (see e.g.\ \cite[XVI, \S 2]{Lan87}) that around
the fixed point $x_{0}$,
\begin{equation}
f(x)=f(x_{0})+Df(x_{0})(x-x_{0})+\left\Vert x-x_{0}\right\Vert o(x-x_{0})
\label{Eq:Jacob}%
\end{equation}
where $o(y)\rightarrow0$ as $y\rightarrow0$. Let $\vec{v}\ $be any
vector in $\mathbb{R}^{n}$, and let $\alpha$ be a positive real
number. Then the last equation yields
(taking $x-x_{0}=\alpha\vec{v}$)%
\begin{align}
Df(x_{0})(\alpha\vec{v})  &  =f(x_{0}+\alpha\vec{v})-f(x_{0})-\left\Vert
\alpha\vec{v}\right\Vert o(\alpha\vec{v})\text{ \ \ }\Rightarrow\nonumber\\
\alpha\left\Vert Df(x_{0})\vec{v}\right\Vert  &  \leq\left\Vert f(x_{0}%
)-f(x_{0}+\alpha\vec{v})\right\Vert +\alpha\left\Vert \vec{v}\right\Vert
\left\Vert o(\alpha\vec{v})\right\Vert \text{ \ \ }\Rightarrow\nonumber\\
\left\Vert Df(x_{0})\vec{v}\right\Vert  &  \leq\frac{\lambda\left\Vert
\alpha\vec{v}\right\Vert }{\alpha}+\left\Vert \vec{v}\right\Vert \left\Vert
o(\alpha\vec{v})\right\Vert \text{ \ \ }\Rightarrow\nonumber\\
\left\Vert Df(x_{0})\vec{v}\right\Vert  &  \leq\left(\lambda + \left\Vert
o(\alpha\vec{v})\right\Vert\right)\left\Vert
\vec{v}\right\Vert \nonumber
\end{align}
The last inequality must be true for all positive $\alpha$. Since the left-hand side of the inequality does not depend on $\alpha$, when $\alpha\to 0$ we get
\begin{align}
\left\Vert Df(x_{0})\vec{v}\right\Vert  &  \leq\lambda\left\Vert \vec
{v}\right\Vert \label{Eq;hyperbolic}%
\end{align}
The last inequality implies that no eigenvalue of $Df(x_{0})$ can
have absolute value$\ $larger than $\lambda<1$. In particular, this
implies that the point $x_{0}$ is a hyperbolic sink for the map $f$.
\end{proof}

Using these two results, we can now prove Theorem \ref{Th:Main_discrete}.

\begin{proof}
[of Theorem \ref{Th:Main_discrete}]Let $M$ be a universal
Turing machine. Suppose, without loss of generality, that $M$ admits
only one halting configuration $s\in\mathbb{N}^{3}$. Since
$s\in\mathbb{N}^{3}$, $s$ is a computable real number. Let $f_{M}$
be the transition function of $M$ and let $f$ be the map that
simulates $M$ according to Theorem \ref{Th:simulation}. Then $s$ is
a fixed point of $f$ and, by Theorem \ref{Th:simulation} and Lemma
\ref{Lemma:hyperbolic}, it is a hyperbolic sink. Denote as
$W_{final}$ the basin of attraction of $s$. Let
$x_{0}\in\mathbb{N}^{3}$ be an initial configuration of $M$. We note
that $M$ halts on
$x_{0}$ iff $x_{0}\in W_{final}$. Moreover, from Theorem \ref{Th:simulation}%
,\ any trajectory starting at a point in $B(x_{0},1/4)=\{z\in\mathbb{R}%
^{3}|\left\Vert x_{0}-z\right\Vert <1/4\}$ will also converge to $s$ if $M$
halts on the initial configuration $x_{0}$ and will not converge to $s$ if $M$
does not halt on $x_{0}$. In other words, $B(x_{0},1/4)\subseteq W_{final}$ if
$M$ halts on $x_{0}$ and $B(x_{0}, 1/4)\subseteq\mathbb{R}^{3}-W_{final}$ if
$M$ does not halt on $x_{0}$.

We recall that $W_{final}$ is an open subset of $\mathbb{R}^{3}$. Now suppose
that $W_{final}$ is a computable set. Then the distance function
$d_{\mathbb{R}^{3}\backslash W_{final}}$ is computable. Therefore we can
compute $d_{\mathbb{R}^{3}\backslash W_{final}}(x_{0})$ with a precision of
$1/10$ yielding some rational $q$. We observe that either $d_{\mathbb{R}%
^{3}\backslash W_{final}}(x_{0})=0$ if $x_{0}\notin W_{final}$, or else
$d_{\mathbb{R}^{3}\backslash W_{final}}(x_{0})\geq1/4$ if $x_{0}\in W_{final}%
$. In the first case, $q\leq1/10$, while in the second case, $q\geq
1/4-1/10=3/20$. Now we have an algorithm that solves the halting problem: on
initial configuration $x_{0}$, compute $d_{\mathbb{R}^{3}\backslash W_{final}%
}(x_{0})$ within a precision of $1/10$ yielding some rational $q$. If
$q\leq1/10$, then $M$ does not halt on $x_{0}$; if $q\geq3/20$, then $M$ halts
on $x_{0}$. This is of course a contradiction to the fact that the halting
problem is non-computable. Therefore $W_{final}$ cannot be computable.
\end{proof}


\subsection{Brief overview of the proof}
From the previous section, it is clear that what remains to be done is to prove Theorem \ref{Th:simulation}. Since the result in this theorem is similar to the simulation result of TMs with analytic maps of \cite{GCB08}, it is natural that our proof uses techniques similar to those used in that paper. The main difference is that in each iteration of the map we require that \eqref{Eq:Th_main} holds for Theorem \ref{Th:simulation}, while in \cite{GCB08} only the weaker condition \eqref{Eq:old_map_simulation} is required.

The problem in obtaining \eqref{Eq:Th_main} from \eqref{Eq:old_map_simulation} is that in the construction used in \cite{GCB08}, the error $\left\Vert f(x)-f(y)\right\Vert$ depends not only on the initial error $\left\Vert x-y\right\Vert$ but also \emph{on the magnitude of} $x$. An involved construction was used in \cite{GCB08} to ensure that, despite this dependence, the magnitude of the error $\left\Vert f(x)-f(y)\right\Vert$ would not exceed 1/4 (this value is fixed \textit{a priori}).

Here we have to improve the previous construction so that $\left\Vert f(x)-f(y)\right\Vert$ is not bounded by an error given a \textit{a priori}, but rather by the dynamic quantity $\left\Vert x-y\right\Vert$. This is achieved by introducing a new function $\xi_{3}$ (see Section \ref{Sec:Preliminaries}) which will then be used to modify the problematic step 5 (this step is at the source of the problem mentionned above) in Section \ref{Proof:Simulation}.

\subsection{Some special functions used in the
construction\label{Sec:Preliminaries}}



We need to construct a map $f$ that satisfies Theorem
\ref{Th:simulation}. In particular, $f$ must satisfy condition
(\ref{Eq:Th_main}). This can be achieved by requiring the map to
\textquotedblleft contract\textquotedblright\ around integer values
with the help of several special functions: $\sigma, l_{2},
\xi_{2}$, and $\xi_{3}$. These functions are described below and all
of them share the common property of being contractions either
around all integers or around some particular integers.

The first special function, $\sigma:\mathbb{R\rightarrow R}$,
is defined by (cf.~Fig.~\ref{fig:funcao_sigma})%
\begin{equation}
\sigma(x)=x-0.2\sin(2\pi x) \label{3:sigma}%
\end{equation}
The function $\sigma$ is a uniform contraction in a neighborhood of integers as the following
proposition shows.%
\begin{figure}[ptb]
\begin{center}
\includegraphics[width=0.6\textwidth]{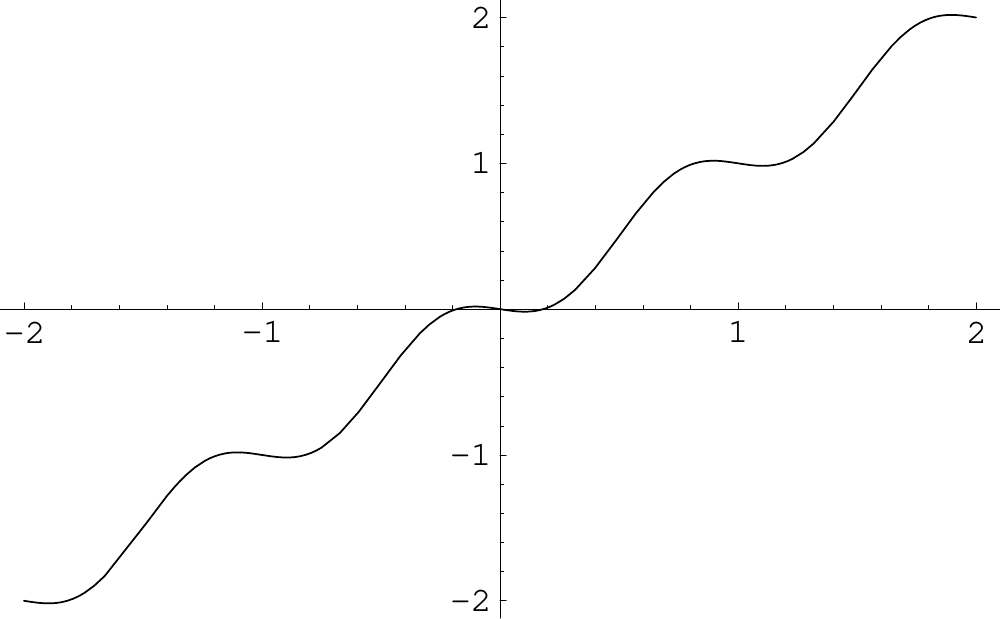}
\end{center}
\caption{Graphical representation of the function $\sigma$.}
\label{fig:funcao_sigma}
\end{figure}%

\begin{proposition}
(\cite{GCB08}) \label{Prop:sigma}Let $\varepsilon\in\lbrack0,1/4).$
Then there is some contracting factor
$\lambda_{1/4}\in(0,1)$ such that for any
$\delta\in\lbrack-1/4,1/4]$, $|\sigma(n+\delta
)-n|<\lambda_{1/4}\delta$ for all $n\in\mathbb{Z}$.
\end{proposition}

For instance (see \cite{GCB08}), we can take $\lambda_{1/4}=0.4\pi
-1\approx0.256637$. It follows from the proposition that for
any $n\in\mathbb{Z}$, every point $x\in B(n,1/4)$ will converge to
$n$ at a rate of $\lambda_{1/4}$ under the application of $\sigma$:
\[
|\sigma^{[k]}(x)-n|<\lambda^{k}_{1/4}\delta
\]
where $x=n+\delta$, $|\delta|<1/4$, and $\sigma^{[k]}(x)$ is the kth
iterate of $\sigma$. Note that the contraction rate is fixed a
priori (it depends on $\varepsilon$) and\ it is the same for all
$n\in\mathbb{Z}$. In some situations, this \textquotedblleft fixed a
priori\textquotedblright\ is undesirable, for we may need to specify
how fast a point of $B(n,1/4)$ should converge to $n$ under the
application of some special function. This is where the function
$l_{2}$ enters. However, this comes at a cost: the \textquotedblleft
dynamic\textquotedblright\ contraction rate is only valid around the
integers 0 and 1. In comparison, $\sigma$ allows a static
contraction rate around every integer $n\in\mathbb{Z}$.

Let us introduce the second special function $l_{2}$. The following result was
proved in \cite{GCB08}.

\begin{proposition}
\label{Prop:l2}Let $l_{2}:\mathbb{R}^{2}\rightarrow\mathbb{R}$ be given by
$l_{2}(x,y)=\frac{1}{\pi}\arctan(4y(x-1/2))+\frac{1}{2}.$ Suppose also that
$a\in\{0,1\}.$ Then%
\[
|a-l_{2}(\overline{a},y)|<\frac{1}{y}
\]
for any $\overline{a},y\in\mathbb{R}$ satisfying $\left\vert a-\overline
{a}\right\vert \leq1/4$ and $y>0.$
\end{proposition}

Thus, $l_{2}$ enjoys \textquotedblleft dynamic\textquotedblright\ contraction
rates around the integers $0$ and $1$. However, it also has an undesirable
feature: $0$ and $1$ are not fixed points of $l_{2}$. For this reason we
introduce a third special function $\xi_{2}$, which is built upon $l_{2}$. The
function $\xi_{2}$ inherits the same \textquotedblleft
dynamic\textquotedblright\ contraction rates around $0$ and $1$ from $l_{2}$,
but it has both $0$ and $1$ as its fixed points, i.e.\ $\xi_{2}(0,y)=0$ and
$\xi_{2}(1,y)=1$ for $y>0$.

We now describe the function $\xi_{2}$. It is easy to see that if $n$ is an
integer, then%
\begin{equation}
0\leq\frac{\sin^{2}(\pi(n+\varepsilon))}{4}\leq\varepsilon\label{Eq:aux2}%
\end{equation}
for $0\leq\varepsilon<1/2$. We define $\xi_{2}:\mathbb{R}\times\mathbb{R}%
^{+}\rightarrow\mathbb{R}$ by the formula:
\begin{equation}
\xi_{2}(x,y)=l_{2}\left(  x,\frac{4y}{\sin^{2}(\pi x)}\right)  \label{Eq:Eta2}%
\end{equation}
where $\mathbb{R}^{+}$ is the set of all positive real numbers. For
$y>0$ and $x\in\mathbb{Z}$, $\xi_{2}(x, y)$ is defined to be the
analytic continuation of the right-hand side of (\ref{Eq:Eta2}). It
is readily seen that $\xi_{2}(0, y)=0$ and $\xi_{2}(1, y)=1$ for all
$y>0$.

The last special function to be defined is called $\xi_{3}$, which is built
upon $\xi_{2}$. The new function $\xi_{3}$ extends the behavior of $\xi_{2}$
around $0$ and $1$ to include the number $2$.

\begin{proposition}
\label{Lema:Zeta3}Let
$\xi_{3}:\mathbb{R}\times\mathbb{R}^{+}\rightarrow \mathbb{R}$ be
given by
\[
\xi_{3}(x,y)=\xi_{2}((\sigma(x)-1)^{2},3y)\cdot
(2\xi_{2}(x/2,3y)-1)+1
\]
If $a=0,1,2$ and $\left\vert a-\bar{a}\right\vert \leq\varepsilon\leq1/4$,
then for all $y\geq2$,
\end{proposition}

\begin{enumerate}
\item If $\bar{a}=a$, then $\xi_{3}(\bar{a},y)=a$;

\item If $\bar{a}\neq a$, then%
\[
\left\vert \xi_{3}(\bar{a},y)-a\right\vert \leq\frac{\varepsilon}{y}%
\]

\end{enumerate}

\begin{proof} It is easy to verify (1). To prove (2), let us first
note $0<\xi_{2}(x,y)<1$ for all $x\in\mathbb{R}$ and
$y\in\mathbb{R}^{+}$.
Consider the case where $a=0$ and $\overline{a}\in\lbrack-1/4;1/4]$
(i.e. $\varepsilon=1/4$). Then $\left\vert (\sigma(\overline{a})-1)^{2}%
-1\right\vert <1/4$ by Proposition \ref{Prop:sigma}, and by Proposition \ref{Prop:l2},%
\[
1-1/y<\xi_{2}((\sigma(\overline{a})-1)^{2},y)<1
\]
Similarly, we conclude
\[
-1<2\xi_{2}(\overline{a}/2,y)-1<-1+2/y
\]
Since $y\geq2,$ this implies%
\[
-1<\xi_{2}((\sigma(\overline{a})-1)^{2},y)(2\xi_{2}(\overline{a}%
/2,y)-1)<(1-1/y)(-1+2/y)
\]
or%
\[
0<\xi_{2}((\sigma(\overline{a})-1)^{2},y)(2\xi(\overline{a}/2,y)-1)+1<3/y
\]
Hence, for $a=0,$ $|a-\xi_{3}(\overline{a},y)|<1/y.$ The same
argument applies to the cases  $a=1$ and $a=2$.
\end{proof}

\subsection{Simulation of Turing machines with maps\label{Sec:MT}}

\subsubsection{Turing machines\label{Sec:Configurations}}

To construct the function $f$ defined in Theorem \ref{Th:simulation}, we have
to encode the behavior of a given Turing machine as the iteration of $f$. To
do this, we need to know a bit more about TMs.

This subsection provides a more detailed description of Turing
machines. It should be noted though that, to keep this subsection
short, the description given here is relatively brief and adapted to
the contents of this paper. It is not intended to give the intuition
behind the model nor to explain why it naturally relates to
algorithms. The reader interested in such aspects is referred to the
excellent introductory textbook \cite{Sip05}. The Definition of TM
used here is not (but is equivalent to) the standard definition in
the literature.

A Turing machine works on triples (configurations) from $(\Sigma
\cup\{B\})^{\ast}\times(\Sigma\cup\{B\})^{\ast}\times\{1,2,\ldots,m\}$
where $\Sigma$ is a finite and non-empty set (set of symbols),
$B\notin\Sigma$ ($B$ is the blank symbol), $(\Sigma
\cup\{B\})^{\ast}$ denotes the set of all finite sequences of
elements from $\Sigma \cup\{B\}$, and $m\geq2$ is some integer. For
example, if $\Sigma=\{0,1\}$, then $(0,1,1,0)\in\Sigma^{\ast}$,
which is usually written as $0110\in\Sigma^{\ast}$. The set
$\{1,2,\ldots,m\}$ is the set of states. The initial state is $1$
and the halting state is $m$. An input to a TM is an element
$w\in\Sigma^{\ast}$; the corresponding initial configuration is
$(w,B,1)$. The Turing machine then updates the configuration%
\begin{equation}
(a_{n}\ldots a_{0},a_{-1}\ldots a_{-k},q), \quad  a_{i}, a_{-j}\in\Sigma\cup\{B\}\label{Eq:conf}%
\end{equation}
according to some fixed rule that depends only on $a_{0}$ (symbol
read by the head) and the state $q$. Depending only on the value
$a_{0}$ and $q$, the TM will perform the following tasks:

(i) update the state $q$ to a new state $q^{\prime}\in\{1,2,\ldots,m\}$ (it
may be $q^{\prime}=q$);

(ii) change the symbol $a_{0}\in\Sigma\cup\{B\}$ to a new symbol
$a_{0}^{\prime}\in\Sigma\cup\{B\}$ (it may be $a_{0}^{\prime}=a_{0}$);

(iii) after performing step (ii) it may \textquotedblleft move the head to
left\textquotedblright\ (cf.\ \cite{Sip05}) yielding the new configuration%
\[
(a_{n}\ldots a_{1},a_{0}^{\prime}a_{-1}\ldots a_{-k},q^{\prime})
\]
or it may \textquotedblleft move the head to right,\textquotedblright%
\ yielding the new configuration%
\[
(a_{n}\ldots a_{1}a_{0}^{\prime}a_{-1},a_{-2}\ldots a_{-k},q^{\prime})
\]
or it may \textquotedblleft not move the head,\textquotedblright\ yielding the
configuration%
\[
(a_{n}\ldots a_{1}a_{0}^{\prime},a_{-1}a_{-2}\ldots a_{-k},q^{\prime})
\]
When performing step (iii), if we obtain a sequence of zero length
in one of the first two components of the configuration, we replace
it by the symbol $B$.

Configurations are updated step-by-step with these rules until the
state eventually reaches the halting state $m$. In this case, the TM
halts; we have reached a halting configuration.

One can code configurations of Turing machines as elements of $\mathbb{N}^{3}%
$. It suffices to have the number 0 correspond to the symbol $B$ and numbers
$\{1,\ldots,l\}$ to elements of $\Sigma$, where $l=\#\Sigma$ (cardinality of
$\Sigma$). Then the configuration (\ref{Eq:conf}) can be seen as a triple
\begin{equation}
(y_{1},y_{2},q)\in\mathbb{N}^{3} \label{Eq:TripleCode}%
\end{equation}
where
\begin{align}
y_{1}  &  =a_{0}+a_{1}l+...+a_{n}l^{n}\label{3:Tape_encoding}\\
y_{2}  &  =a_{-1}+a_{-2}l+...+a_{-k}l^{k-1}\nonumber
\end{align}
Thus, to define a Turing machine, it suffices to know how to go from
one configuration to the next one. In other words, a Turing machine
$M$ can be defined by its transition function
$f_{M}:\mathbb{N}^{3}\rightarrow\mathbb{N}^{3}$.

\subsubsection{Determining the next action - Interpolation
techniques\label{Sec:3.NextAction}}

Let us fix a Turing machine $M$ in the remainder of Section
\ref{Sec:DiscreteTime}. Recall that our ultimate objective (Theorem
\ref{Th:simulation}) is to build an analytic and non-expanding map
$f:\mathbb{R}^3\to\mathbb{R}^3$ that simulates the behavior of the
Turing machine $M$, or in other words, the restriction of $f$ on
$\mathbb{N}^3$ is the transition function $f_{M}$ of the machine
$M$.

We know that given a configuration (\ref{Eq:conf}), the computation of the
next configuration only depends on $a_{0}$ and $q$. So, when given a
configuration in the form of (\ref{Eq:TripleCode}) as an input to $f$, we have
to extract the values of $a_{0}$ and $q$ in order to find what $M$ (and
therefore $f$) is supposed to do. The value $q$ is readily available as the
last component of (\ref{Eq:TripleCode}). However $a_{0}$ must be extracted
from the first component $y_{1}$.\smallskip

\noindent\textbf{Extracting the symbol }$\boldsymbol{a}_{0}$\textbf{.}
Consider an analytic extension $\bar{\omega}:\mathbb{R}\rightarrow\mathbb{R}$
of the function $g:\mathbb{N}\rightarrow\mathbb{N}$ defined by
$g(n)=n\operatorname{mod}l$ (in the case where the tape alphabet of the TM has
$l$ symbols). It follows from the coding (\ref{3:Tape_encoding}) that
$\bar{\omega}(y_{1})=a_{0}$, i.e., $\bar{\omega}$ extracts $a_{0}$ from
$y_{1}$. We also require $\bar{\omega}$ to be a periodic function, of period
$l$, such that $\bar{\omega}(i)=i$, for $i=0,1,...,l-1$. The function
$\bar{\omega}$ can be constructed by using trigonometric interpolation
(cf.~\cite[pp. 176-182]{Atk89}), which produces an analytic as well as
periodic computable function. For example, if $l=10$, then one can define
$\bar{\omega}$ as follows:%
\begin{equation}
\bar{\omega}(x)=\alpha_{0}+\alpha_{5}\cos(\pi x)+\left(  \sum_{j=1}^{4}%
\alpha_{j}\cos\left(  \frac{j\pi x}{5}\right)  +\beta_{j}\sin\left(
\frac{j\pi x}{5}\right)  \right)  \label{3:Omega}%
\end{equation}
where%
\begin{align*}
\alpha_{0}  &  =9/2\text{, \ \ \ }\alpha_{1}=\alpha_{2}=\alpha_{3}=\alpha
_{4}=-1\text{, \ \ \ }\alpha_{5}=-1/2\\
\beta_{1}  &  =-\sqrt{5+2\sqrt{5}}\text{, \ \ }\beta_{2}=-\sqrt{1+\frac
{2}{\sqrt{5}}}\text{, \ \ }\beta_{3}=-\sqrt{5-2\sqrt{5}}\text{, \ \ }\beta
_{4}=-\sqrt{1-\frac{2}{\sqrt{5}}}%
\end{align*}
The construction ensures that, over integer arguments,
$\bar{\omega}$ gives exact results; but it does not guarantee that
$\bar{\omega}$ is non-expanding around integers, which is needed for
condition (\ref{Eq:Th_main}). To meet  the non-expanding
requirement, we compose $\bar{\omega}$ with the \textquotedblleft
uniform contraction\textquotedblright function $\sigma$ as
follows: Let $K$ be some integer such that%
\[
K\geq\max_{x\in\lbrack0,l]}\left\vert
\bar{\omega}^{\prime}(x)\right\vert =\max
_{x\in\mathbb{R}}\left\vert \bar{\omega}^{\prime}(x)\right\vert
\]
where $\bar{\omega}^{\prime}$ is the derivative of $\bar{\omega}$,
and let $k\in\mathbb{N}$ be such that $K\lambda_{1/4}^{k}\leq1$,
where $\lambda_{1/4}$ is given in Proposition \ref{Prop:sigma} (its
exact value is given right after Proposition \ref{Prop:sigma}). Then
define
\begin{equation}
\omega=\bar{\omega}\circ\sigma^{\lbrack k]} \label{Eq:ImprovedOmega}%
\end{equation}
It follows that if $y_{1}$ is given by (\ref{3:Tape_encoding}), then
$\omega(y_{1})=\bar{\omega}\circ\sigma^{\lbrack k]}(y_{1})=\bar{\omega}(y_1)=a_{0}$
for $\sigma^{\lbrack k]}(y_{1})=y_1$;
thus $\omega$ extracts $a_{0}$ from $y_{1}$. Moreover, $\omega$ has
the desired non-expanding property as shown below: for any $y_{1}$
given by (\ref{3:Tape_encoding}) and $y\in\mathbb{R}$,
\begin{align*}
\left\vert y-y_{1}\right\vert  &  \leq\varepsilon\leq1/4\text{\ \ }%
\Longrightarrow\text{ \ }\left\vert \sigma^{\lbrack
k]}(y)-y_{1}\right\vert
\leq\varepsilon\lambda_{1/4}^{k}\text{\ \ }\Longrightarrow\text{ \ }\\
\left\vert \bar{\omega}\circ\sigma^{\lbrack k]}(y)-\bar{\omega}(y_{1}%
)\right\vert  &  \leq\varepsilon K\lambda_{1/4}^{k}\text{\ \
}\Longrightarrow \text{ \ \ }\left\vert
\omega(y)-\omega(y_{1})\right\vert \leq\varepsilon
\end{align*}

\noindent\textbf{Encoding the next action to be performed.} After
knowing $a_{0}$ and $q$, we now need to encode the next action to
be performed by the  machine $M$, i.e. the new symbol to be
written in the configuration, the next move to be performed, and the
new state. Using Lagrange interpolation we can encode each of these performances
by an analytic function.

Let $Q_{j}, S_{i}:\mathbb{R}\to\mathbb{R}$, $0\leq i\leq l$ and $1\leq j\leq m$,
be the functions defined as follows: %
\[
Q_{j}(x)=\prod_{\substack{k=1\\k\neq j}}^{m}\frac{(x-k)}{(j-k)}\text{,
\ \ \ }S_{i}(x)=\prod_{\substack{k=0\\k\neq i}}^{l}\frac{(x-k)}{(i-k)}
\]
Note that%
\begin{align*}
Q_{j}(x)  &  =\left\{
\begin{array}
[c]{l}%
0,\text{ if }x=1,...,j-1,j+1,...,m\\
1,\text{ if }x=i
\end{array}
\right.  \text{ \ \ and \ \ }\\
S_{i}(x)  &  =\left\{
\begin{array}
[c]{l}%
0,\text{ if }x=0,...,i-1,i+1,...,l\\
1,\text{ if }x=i
\end{array}
\right.
\end{align*}

Suppose that on symbol $i$ and state $j$, the state of the next configuration
is $q_{i,j}.$ Then the state that follows symbol $a_{0}$ and state $q$ is
given by%
\begin{equation}
\overline{q_{next}}(a_{0},q)=\sum_{i=0}^{l}\sum_{j=1}^{m}S_{i}(a_{0}%
)Q_{j}(q)q_{i,j} \label{3:Interpolacao_Polinomial}%
\end{equation}
A similar procedure can be used to determine the next symbol to be
written and the next move. It is easy to see that
$\overline{q_{next}}(a_0, q)$ returns the state of the next
configuration if the machine $M$ is in the state $q$ reading the
symbol $a_0$.
But again $\overline{q_{next}}$ may fail to
be non-expanding. 
This problem is dealt with similarly as in the previous case. Let $K$ be some integer
such that%
\[
K\geq\max_{\substack{s\in\lbrack-1,l+1]\\q\in\lbrack0,m+1]}}\left\Vert
\nabla\overline{q_{next}}(s,q)\right\Vert _{2}%
\]
where $\nabla\overline{q_{next}}$ is the gradient of $\overline{q_{next}}$ and
$\left\Vert \cdot\right\Vert _{2}$ is the Euclidean norm,  and let
$k\in\mathbb{N}$ be such that $K\lambda_{1/4}^{k}\leq1$ (again $\lambda_{1/4}$
is given by Proposition \ref{Prop:sigma}). Now we define $q_{next}:\mathbb{R}^2\to\mathbb{R}$
as follows:  $q_{next}(x_{1}, x_{2})=\overline{q_{next}}(\sigma^{\lbrack k]}(x_{1}), \sigma^{\lbrack k]}(x_{2}))$
for all $x_{1}, x_{2}\in\mathbb{R}$.
The function $q_{next}%
$\ satisfies the following two conditions: for all $i=0,\ldots,l$,
$j=1,\ldots,m$, $q_{next}(i,j)=\overline{q_{next}}(\sigma^{\lbrack
k]}(i),\sigma^{\lbrack k]}(j))=q_{i,j}$ and it is non-expanding as shown
below:
\begin{align*}
\left\Vert (x_{1},x_{2})-(i,j)\right\Vert  &  \leq\varepsilon\leq
1/4\text{\ \ }\Longrightarrow\text{ \ }\left\Vert (\sigma^{\lbrack k]}%
(x_{1}),\sigma^{\lbrack k]}(x_{2}))-(i,j)\right\Vert \leq\varepsilon
\lambda_{1/4}^{k}\text{\ \ }\Longrightarrow\text{ \ }\\
\left\Vert \overline{q_{next}}(\sigma^{\lbrack k]}(x_{1}),\sigma^{\lbrack
k]}(x_{2}))-\overline{q_{next}}(i,j)\right\Vert  &  \leq\varepsilon
K\lambda_{1/4}^{k}\text{\ \ }\Longrightarrow\text{ \ \ }\left\Vert
q_{next}(x_{1},x_{2})-q_{next}(i,j)\right\Vert \leq\varepsilon
\end{align*}

\subsection{Proof of Theorem \ref{Th:simulation}\label{Proof:Simulation}}

We need to construct an analytic and non-expanding map $f$ that simulates the behavior
of the machine $M$ (i.e. $f$ is an extension of the transition function $f_{M}$ of $M$) and satisfies the condition (\ref{Eq:Th_main}).
As a first
step, we construct a map $\tilde{f}:\mathbb{R}^{3}\rightarrow\mathbb{R}^{3}$
that has all the properties of the function $f$ in Theorem \ref{Th:simulation}%
, except that $\tilde{f}$ satisfies the condition (\ref{Eq:Th_main})
with $\lambda=1$ rather than the desired $\lambda\in(0, 1)$. To
remedy this deficiency, we make use of the \textquotedblleft uniform
contraction\textquotedblright $\ $function $\sigma$ again and let
$f(x_{1}, x_{2}, x_{3})=\tilde {f}(\sigma(x_{1}), \sigma(x_{2}),
\sigma(x_{3}))$ for all $x=(x_{1}, x_{2}, x_{3})\in\mathbb{R}^3$.
Then $f$ would satisfy the condition (\ref{Eq:Th_main}) with
$\lambda=\lambda_{1/4}=0.4\pi-1$. Since $\tilde{f}$ and $\sigma$ are
both analytic, so is $f$. Moreover, if $\tilde{f}$ simulates the
machine $M$, then so does $f$ because $f(n_{1}, n_{2}, n_{3})=\tilde
{f}(\sigma(n_{1}, \sigma(n_{2}), \sigma(n_{3}))=\tilde{f}(n_{1},
n_{2}, n_{3})$ for all $(n_{1}, n_{2}, n_{3})\in\mathbb{N}^3$. Thus
$f$ satisfies all conditions of Theorem \ref{Th:simulation} and the
proof of Theorem \ref{Th:simulation} is then complete.

The remainder of this subsection is devoted to construction of the
function $\tilde{f}$. We assume $0\leq\varepsilon\leq1/4$ throughout
subsection \ref{Proof:Simulation}.

\begin{enumerate}
\item \textbf{Extract the symbol }$\boldsymbol{a}_{0}$\textbf{.} Let $a_{0}$
be the symbol being actually read by the machine $M$. Then
$\omega(y_{1})=a_{0},$ where $\omega$ is given by (\ref{Eq:ImprovedOmega}).
Moreover, $\omega$ is analytic and non-expanding around integers, as we have seen.

\item \textbf{Encode the next state. }The map $q_{next}$ returns the next
state and 
is non-expanding around meaningful integer vectors, where an integer vector
$(s, q)$ is called meaningful if $s$ codes a symbol and $q$ codes a state.

\item \textbf{Encode the symbol to be written on the tape.} Similarly to
the state, we can define a map $s_{next}:\mathbb{R}^{2}\rightarrow\mathbb{R}$
such that $s_{next}(i,j)$ returns the next symbol to be written on the tape if
the machine $M$ is reading symbol $i$ and is in state $j$. This map is
non-expanding around meaningful integer vectors $(i,j)$.

\item \textbf{Encode the direction of the move for the head.} Let $h$
denote the direction of the move of the head, where $h=0$ indicates a move to
the left, $h=1$ a \textquotedblleft no move\textquotedblright, and
$h=2$ a move to the right. Then, similarly to the state, we can define
a map $h_{next}:\mathbb{R}^{2}\rightarrow\mathbb{R}$ such that $h_{next}(i,j)$
returns the next move to be written on the tape if the machine $M$ is
reading symbol $i$ and is in state $j$. This map is non-expanding around
meaningful integer vectors $(i,j)$.

\item \textbf{Update the tape contents.} In the absence of error, the
\textquotedblleft next value\textquotedblright\ of $y_{1}$, $\overline
{y_{1}^{next}}$, is given by Lagrange interpolation as a function of $y_{1}$,
$y_{2}$, and $q$ as follows (to simplify notation, let us use $s_{next}$ and $h_{next}$
to represent $s_{next}(\omega(y_1), q)$ and $h_{next}(\omega(y_1), q)$ respectively):
\begin{gather}
\overline{y_{1}^{next}}(y_{1},y_{2},s_{next},h_{next})=(l\cdot(y_{1}%
+s_{next}-\omega(y_{1}))+\omega(y_{2}))\frac{(1-h_{next})(2-h_{next})}%
{2}\label{Eq:y1_next}\\
+(y_{1}+s_{next}-\omega(y_{1}))h_{next}(2-h_{next})+\frac{y_{1}-\omega(y_{1}%
)}{l}\frac{h_{next}(1-h_{next})}{-2}\nonumber
\end{gather}
When the head moves left, doesn't move, or moves right, the first, second, or
third term on the right-hand side gives the \textquotedblleft next
value\textquotedblright\ of $y_{1}$, respectively. A function $\overline
{y_{2}^{next}}$ giving the \textquotedblleft next value\textquotedblright\ of
$y_{2}$ can be constructed in a similar way.
\end{enumerate}

When $y_{1}$ and $y_{2}$ are given by (\ref{3:Tape_encoding}) and
$q$ is a state (such $y_{1}$, $y_{2}$, and $q$ are called meaningful
integers), $\overline{y_{1}^{next}}(y_1, y_2, s_{next},h_{next})$
returns the exact value of $y_{1}$ in (\ref{3:Tape_encoding}) for
the next configuration.
Unfortunately, there is a problem with this function:
$\overline{y_{1}^{next}}$ is not only expanding, but also in a way
harder to deal with than in the case of $\overline{q_{next}}$. Here
is why: let $A: \mathbb{R}^2\to\mathbb{R}$ be the simple product
defined by $A(x, y)=x\cdot y$. Then for any $\varepsilon, \delta>0$,
$A(x+\varepsilon,y+\delta)-A(x,y)=\varepsilon\delta+\varepsilon
y+\delta x$. Thus, as $x$ and $y$ grow, the difference
$A(x+\varepsilon,y+\delta)-A(x, y)$ also grows, even when
$\varepsilon$ and $\delta$ are held constant. Since each term on the
right-hand side of (\ref{Eq:y1_next}) is a product containing
$y_{1}$, and $y_{1}$ can grow arbitrarily large, it follows that the
map $\overline{y_{1}^{next}}$ is expanding around positive integers.
Moreover, the rate of expansion is non-uniform because the rate
depends on the module of $y_{1}$. Due to this non-uniform expansion
nature of $\overline{y_{1}^{next}}$, in comparison with the
deduction of $q_{next}$ from $\overline{q_{next}}$, more work is
needed in order to build a non-expanding \textquotedblleft next
value\textquotedblright\ function $y_{1}^{next}$ from
$\overline{y_{1}^{next}}$.


So we need to build a map $y^{next}_{1}$ satisfying the following two
conditions: for any $y_{1}, y_{2}, s_{next}, h_{next}\in\mathbb{N}$,
\ $\overline{y_{1}},\overline{y_{2}},\overline{s_{next}},\overline{h_{next}%
}\in\mathbb{R}$ with $\left\Vert (y_{1},y_{2},s_{next},h_{next})-\right.
\newline\left. (\overline{y_{1}},\overline{y_{2}},\overline{s_{next}%
},\overline{h_{next}})\right\Vert < \varepsilon\leq1/4$,%
\[
y_{1}^{next}(y_{1},y_{2},s_{next},h_{next})=\overline{y_{1}^{next}}%
(y_{1},y_{2},s_{next},h_{next})
\]
and
\begin{equation}
\left\Vert y_{1}^{next}(y_{1},y_{2},s_{next},h_{next})-y_{1}^{next}%
(\overline{y_{1}},\overline{y_{2}},\overline{s_{next}},\overline{h_{next}%
})\right\Vert \leq\varepsilon\label{Eq:y1_not_expanding}%
\end{equation}
that is, $y_{1}^{next}$ gives the \textquotedblleft next
value\textquotedblright of $y_{1}$ as $\overline{y^{next}_{1}}$ does and
$y^{next}_{1}$ is non-expanding. To build this map, we use the interpolation
idea behind (\ref{Eq:y1_next}), which of course has to be improved for the
reasons mentioned above.

Let us write $y_{1}^{next}=P_{stay}+P_{left}+P_{right}$, where $P_{stay}%
(y_{1},y_{2},s_{next},h_{next})$ is a term that is non-zero only when $h=1$,
i.e.\ when the head does not move, and its value is the \textquotedblleft next
value\textquotedblright\ of $y_{1}$ in this case. The other terms $P_{left},$
and $P_{right}$ are defined similarly. If we can construct the parcels
$P_{left},P_{stay},P_{right}$ such that%
\begin{align}
\left\Vert P_{left}(y_{1},y_{2},s_{next},h_{next})-P_{left}(\overline{y_{1}%
},\overline{y_{2}},\overline{s_{next}},\overline{h_{next}})\right\Vert  &
\leq\varepsilon/3\label{Conditions_P}\\
\left\Vert P_{stay}(y_{1},y_{2},s_{next},h_{next})-P_{stay}(\overline{y_{1}%
},\overline{y_{2}},\overline{s_{next}},\overline{h_{next}})\right\Vert  &
\leq\varepsilon/3\nonumber\\
\left\Vert P_{right}(y_{1},y_{2},s_{next},h_{next})-P_{right}(\overline{y_{1}%
},\overline{y_{2}},\overline{s_{next}},\overline{h_{next}})\right\Vert  &
\leq\varepsilon/3\nonumber
\end{align}
then (\ref{Eq:y1_not_expanding}) is satisfied, i.e.,\ $y_{1}^{next}$ is
non-expanding. Let us show how to obtain a parcel $P_{stay}$ with the above
property. The same argument can be applied to the parcels $P_{left}$ and $P_{right}$.

The parcel $P_{stay}$ is constructed by modifying the term
\[
(y_{1}+s_{next}-\omega(y_{1}))h_{next}(2-h_{next})
\]
in (\ref{Eq:y1_next}) (we recall that this term gives the \textquotedblleft
next value\textquotedblright of $y_{1}$ if the head doesn't move). We intend
to define $P_{stay}$ to be the product of two functions $C(\cdot,\cdot)$ and
$D(\cdot)$, where $C(y, z)=\sigma^{\lbrack2]}(y)+\sigma^{\lbrack2]}%
(z)-\sigma^{\lbrack2]}(\omega(y))\ $ and $D(x)=x(2-x)$. Since $0<\varepsilon<1/4$
and $\left\vert \left(  y_{1}+s_{next}%
-\omega(y_{1})\right)  -\left(  \overline{y_{1}}+\overline{s_{next}}%
-\omega(\overline{y_{1}})\right)  \right\vert \leq3\varepsilon$, the following
holds true:
\[
\left\vert C(\overline{y_{1}},\overline{s_{next}})-\left(  y_{1}%
+s_{next}-\omega(y_{1})\right)  \right\vert \leq3\varepsilon\lambda_{1/4}%
^{2}<\frac{\varepsilon}{5}%
\]
In other words, the function $C(\cdot, \cdot)$ is non-expanding. Thus, to
guarantee that the product $C(\cdot,\cdot)D(\cdot)$ is non-expanding, it suffices
to find the condition that ensures non-expansiveness of the function $D$.
To this end, we observe that
\[
\max_{x\in\lbrack-1,3]}\left\vert D^{\prime}(x)\right\vert =4
\]
and it then follows that $\left\vert D(w)-D(h_{next})\right\vert
\leq4\delta$ provided $\left\vert w-h_{next}\right\vert
\leq\delta\leq1/4$ (recall that $h_{next}\in\{ 0, 1, 2\}$).
In particular, since $\left\vert C(\overline{y_{1}%
},\overline{s_{next}})\right\vert \leq\overline{y_{1}}+l$ and $\left\vert
D(h_{next})\right\vert \leq1$, the following estimate holds:%
\begin{gather}
\left\vert C(\overline{y_{1}},\overline{s_{next}})D(w)-C(y_{1},s_{next}%
)D(h_{next})\right\vert \leq\label{Eq:aux1}\\
\left\vert C(\overline{y_{1}},\overline{s_{next}})(D(w)-D(h_{next}%
))\right\vert +\left\vert D(h_{next})(C(\overline{y_{1}},\overline{s_{next}%
})-C(y_{1},s_{next}))\right\vert \leq\nonumber\\
(\overline{y_{1}}+l)4\delta+\frac{\varepsilon}{5}\nonumber
\end{gather}
Therefore, if we can compute some $w=\theta(\overline{h_{next}},\overline
{y_{1}})$ satisfying $\left\vert w-h_{next}\right\vert \leq\delta$ with
\begin{equation}
(\overline{y_{1}}+l)4\delta\leq\frac{\varepsilon}{10},\text{ \ \ }%
\mbox{or equivalently,}\text{ \ \ }\delta\leq\frac{\varepsilon}{40(\overline{y_{1}}+l)}
\label{Eq:aux3}%
\end{equation}
then we can define $P_{stay}$ as follows:
\[
P_{stay}(\overline{y_{1}},\overline{y_{2}},\overline{s_{next}},\overline
{h_{next}})=C(\overline{y_{1}},\overline{s_{next}})D(\theta(\overline
{h_{next}},\overline{y_{1}}))
\]
From (\ref{Eq:aux1}) and (\ref{Eq:aux3}) it follows that $P_{stay}$
satisfies the second condition of (\ref{Conditions_P}); thus $P_{stay}$ is non-expanding.

It therefore remains to define a function $\theta$ that has the following
property:
\[
\left\vert \theta(\overline{h_{next}},\overline{y_{1}})-h_{next}\right\vert
\leq\frac{\varepsilon}{40(\overline{y_{1}}+l)}%
\]
From Proposition \ref{Lema:Zeta3} and the facts that $\left\vert
\overline {h_{next}}-h_{next}\right\vert \leq\varepsilon$ and
$h_{next}\in\{0,1,2\}$, it becomes clear that we can make use of the
special function $\xi_{3}$ to get
the desired function $\theta$ as follows:%
\[
\theta(\overline{h_{next}},\overline{y_{1}})=\xi_{3}(\overline{h_{next}%
},40(\overline{y_{1}}+l))
\]
With $\theta$ defined as above, $P_{stay}$ is non-extending. Moreover, we
observe that
\[
P_{stay}(y_{1},y_{2},s_{next},h_{next})=C(y_{1},s_{next})D(\theta
(h_{next},y_{1}))=\overline{y_{1}^{next}}(y_{1},y_{2},s_{next},h_{next})
\]
since $\sigma^{\lbrack2]}(n)=n$ for any $n\in\mathbb{N}$ and $\xi_{3}%
(h_{next},40(y_{1}+l))=h_{next}$; thus the first condition imposed on
$y_{1}^{next}$ is also satisfied.

\medskip

We now have all the pieces needed to build the map $\tilde{f}:\mathbb{R}%
^{3}\mathbb{\rightarrow R}^{3}$ mentioned at the beginning of this section. It
is defined as follows:%
\begin{align*}
\tilde{f}(y_{1},y_{2},q)  &  =(y_{1}^{next}(y_{1},y_{2},s_{next}(\omega
(y_{1}),q),h_{next}(\omega(y_{1}),q)),\\
&  y_{2}^{next}(y_{1},y_{2},s_{next}(\omega(y_{1}),q),h_{next}(\omega
(y_{1}),q)),\\
&  q_{next}(\omega(y_{1}),q))
\end{align*}
It is also easy to see that $\tilde{f}$ and therefore $f$ are both
computable (they are defined by composing standard computable
functions and some of their analytic continuations, and therefore
are computable; see \cite{PR89}, \cite{BHW08}).

\section{The continuous-time case}\label{SecContinuousTime}

Now that we have proved Theorem \ref{Th:Main_discrete}, we proceed
to prove Theorem \ref{Th:Main_continuous}. To construct the system
$y^{\prime}=g(y)$ of Theorem \ref{Th:Main_continuous}, the idea is
to embed the discrete-time system defined in Theorem
\ref{Th:Main_discrete} in a continuous-time system, as we noted in
Section \ref{Sec:Results roadmap}. We strongly suggest that the
reader  use Section \ref{Sec:Results roadmap}\ as a guide for the
next sections. Section \ref{Sec:Branickyinfty} is basically material from \cite{Bra95}, 
\cite{CMC00}, \cite{CM01}, \cite{Cam02b}, which is used as a building block for Section \ref{Sec:Branickyanalytic} which is material from \cite{GCB08}. Those sections will provide enough background to prove the completely new results of Sections \ref{Sec:Hyper}, \ref{Sec:Simulate_TM}, and \ref{Sec:Hyperbolic}.

\subsection{Simulations of Turing machines with ODEs - non-analytic case\label{Sec:Branickyinfty}}


In this subsection we show how to iterate a map from integers to integers with
smooth (but non-analytic) ODEs. By a smooth ODE we mean an ODE%
\begin{equation}
y^{\prime}=g(t,y) \label{3:ODE_f}%
\end{equation}
where $g$ is of class $C^{k}$ for $1\leq k\leq\infty$. This
construction will be refined later (Subsection
\ref{Sec:Branickyanalytic}) to include the analytic case. For the
non-analytic case, we make use of the construction presented\ by
Branicky in \cite{Bra95}, but following the approach of
\cite{CMC00}, \cite{CM01}, \cite{Cam02b}, \cite{GCB08}, which allows
iteration of a map with a smooth ODE.
We say that an ODE $y^{\prime}=g(t,y)$ iterates a map $f$ if $\vert
y(k, y_0)-f^{[k]}(y_0)\vert <\gamma$, $k\in\mathbb{N}$, for some
$\gamma>0$, where $y(t, y_0)$ is the solution to the initial-value
problem $y^{\prime}=g(t,y)$ and $y(0)=y_0$. Note that if we iterate
the map  given in Theorem \ref{Th:simulation} with a smooth ODE,
then this ODE simulates a TM simultaneously.

The construction that allows one to iterate a map with a smooth ODE is given
below (Construction \ref{Construction_Branicky}). It is preceded by two
auxiliary results (Constructions \ref{Construction_Target}
and\ \ref{Construction_r}). For simplicity, the constructions are presented on $\mathbb{R}$.

The first construction, Construction \ref{Construction_Target},
presents an ODE whose trajectories target a given value at a
specified time, whatever the initial states. This is needed later,
for example, to update $y_{A}$ to the targeted value $f(y_{B})$ at
time $t=1/2$, as the reader may recall from Section \ref{Sec:Results
roadmap}.

\begin{notation}
\label{Construction_Target}Consider a point $b\in\mathbb{R}$ (the
\emph{target}), some $\gamma>0$ (the \emph{targeting error}), and
times $t_{0}\geq0$ (\emph{departure time}) and $t_{1}$
(\emph{arrival time}), with $t_{1}>t_{0}$. Then we obtain an
initial-value problem (IVP) defined with an ODE (\ref{3:ODE_f}),
where $g:\mathbb{R}^{2}\rightarrow\mathbb{R}$, such that the
solution $y$ satisfies
\begin{equation}
\left\vert y(t_{1})-b\right\vert <\gamma\label{3:Target_condition}%
\end{equation}
independent of the initial condition $y(t_{0})\in\mathbb{R}$.
\end{notation}

Let $\phi:\mathbb{R}\rightarrow\mathbb{R}_{0}^{+}$ be some function satisfying
$\int_{t_{0}}^{t_{1}}\phi(t)dt>0$ and consider the following ODE%
\begin{equation}
y^{\prime}=c(b-y)^{3}\phi(t) \label{3:ODE_target}%
\end{equation}
where $c>0$. We note that $\mathbb{R}_{0}^{+}=[0, \infty)$ is contained in the
interval of existence of any solution to (\ref{3:ODE_target}), regardless of
the initial states. There are two cases to consider:\ (i) $y(t_{0})=b$, (ii)
$y(t_{0})\neq b$. In the first case, the solution is given by $y(t)=b$ for all
$t\in\mathbb{R}$ and (\ref{3:Target_condition}) is trivially satisfied. For
the second case, note that (\ref{3:ODE_target}) is a separable equation, which
can be explicitly solved as follows:%
\begin{align}
\frac{1}{(b-y(t_{1}))^{2}}-\frac{1}{(b-y(t_{0}))^{2}}  &  =2c\int_{t_{0}%
}^{t_{1}}\phi(t)dt\text{ \ \ }\Longrightarrow\nonumber\\
\frac{1}{2c\int_{t_{0}}^{t_{1}}\phi(t)dt}  &  >(b-y(t_{1}))^{2}
\label{Eq:target}%
\end{align}
Hence, (\ref{3:Target_condition}) is satisfied if $c$ satisfies $\gamma
^{2}\geq(2c\int_{t_{0}}^{t_{1}}\phi(t)dt)^{-1}$ i.e.,~if%
\begin{equation}
c\geq\frac{1}{2\gamma^{2}\int_{t_{0}}^{t_{1}}\phi(t)dt} \label{3:Target_Def_c}%
\end{equation}

Notice that, in the construction above, there is an approximation
error $\gamma$ when approaching the target. This error can be
removed using the function $r$ defined in the following
construction.

\begin{notation}
\label{Construction_r} We obtain an IVP defined with an ODE
(\ref{3:ODE_f}), where $g:\mathbb{R}^{2}\rightarrow\mathbb{R}$, such
that the solution $r$ satisfies the condition below (cf.~Fig.~\ref{fig:functionR}):%
\begin{equation}
r(x)=j\text{ \ \ \ whenever }x\in\lbrack j-1/4,j+1/4]\text{ for all }%
j\in\mathbb{Z}\text{} \label{3:Property_r}%
\end{equation}

\end{notation}

This particular function $r:\mathbb{R}\rightarrow\mathbb{R}$ is needed for the
following reason. Suppose that, in Construction \ref{Construction_Target},
$0<\gamma<1/4$ and $b\in\mathbb{N}$. Then $r(y(t_{1}))=b$, i.e.,~$r$ corrects
the error present in $y(t_{1})$ when approaching an integer value $b$. 
\begin{figure}[ptb]
\begin{center}
\includegraphics[width=0.6\textwidth]{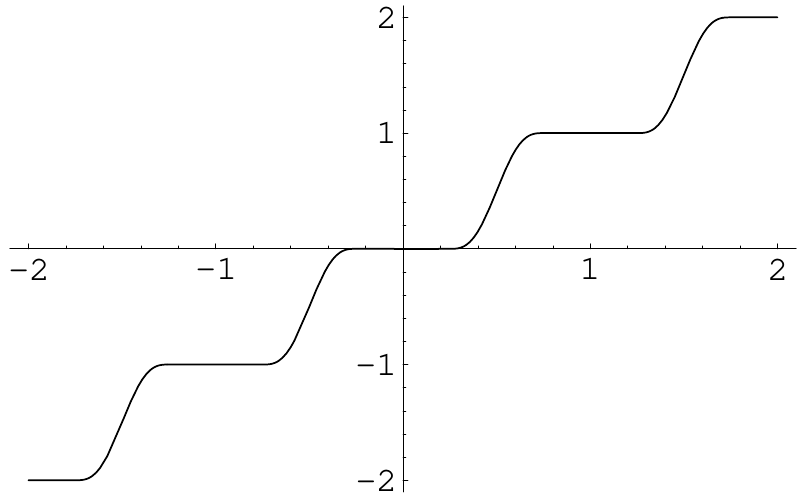}
\end{center}
\caption{Graphical representation of the function $r$.}
\label{fig:functionR}
\end{figure}%

Now the construction. First let $\theta_{k}:\mathbb{R}\rightarrow\mathbb{R}$,
$k\in\mathbb{N}-\{0,1\}$, be the function defined by
\[
\theta_{k}(x)=0\text{ if }x\leq0,\text{ \ \ }\theta_{k}(x)=x^{k}\text{ if
}x>0
\]
For $k=\infty$ define%
\[
\theta_{k}(x)=0\text{ if }x\leq0,\text{ \ \ }\theta_{k}(x)=e^{-\frac{1}{x}%
}\text{ if }x>0
\]
These functions can be seen \cite{CMC00} as a $C^{k-1}$ version of Heaviside's
step function $\theta(x),$ where $\theta(x)=1$ for $x\geq0$ and $\theta(x)=0$
for $x<0$.

With the help of $\theta_{k}$, we define a \textquotedblleft step
function\textquotedblright\ $s:\mathbb{R}\rightarrow\mathbb{R}$, which matches
the identity function on the integers, as follows:%
\[
\left\{
\begin{array}
[c]{l}%
s^{\prime}(x)=\lambda_{k}\theta_{k}(-\sin2\pi x)\\
s(0)=0
\end{array}
\right.
\]
where
\[
\lambda_{k}=\frac{1}{\int_{1/2}^{1}\theta_{k}(-\sin2\pi x)dx}>0
\]
For $x\in\lbrack0,1/2]$, $s(x)=0$ since $\sin2\pi x\geq0$. On
$(1/2,1)$, $s$ strictly increases and satisfies $s(1)=1$. Using the
same argument for $x\in\lbrack j,j+1]$, for all integer $j$, we
conclude that $s(x)=j$ whenever $x\in\lbrack j,j+1/2]$. Then let
$r:\mathbb{N}\rightarrow\mathbb{N}$, $r(x)=s(x+1/4)$. It is easy to
see that $r$ satisfies the condition (\ref{3:Property_r}). Formally,
we should note that, for each $k\in
\mathbb{N\cup\{\infty\}}-\{0,1\}$, we get a different function $r$,
but they all have the same fundamental property
(\ref{3:Property_r}). So we choose to omit the reference to the
index $k$ when defining $r$ (this won't present any problems in
later results).

Now that we can reach a target and remove approximation errors, we can iterate
maps with ODEs.

\begin{notation}
\label{Construction_Branicky}Iterate a map $f_{M}:\mathbb{N}\rightarrow
\mathbb{N}$ with a smooth ODE (\ref{3:ODE_f}).
\end{notation}

Let $f:\mathbb{R}\rightarrow\mathbb{R}$ be an arbitrary smooth extension of
$f_{M}$, and consider the IVP defined with the smooth ODE%
\begin{equation}
\left\{
\begin{array}
[c]{l}%
z_{1}^{\prime}=c_{j,1}(f(r(z_{2}))-z_{1})^{3}\theta_{j}(\sin2\pi t)\\
z_{2}^{\prime}=c_{j,2}(r(z_{1})-z_{2})^{3}\theta_{j}(-\sin2\pi t)
\end{array}
\right.  \label{3:Branicky}%
\end{equation}
and the initial conditions%
\[
\left\{
\begin{array}
[c]{c}%
z_{1}(0)=x_{0}\\
z_{2}(0)=x_{0}%
\end{array}
\right.
\]
where $x_{0}\in\mathbb{N}$. Roughly speaking, the special feature
concerning system (\ref{3:Branicky}) is that one component is
\textquotedblleft active,\textquotedblright while the other is
\textquotedblleft dormant\textquotedblright\ and serves as a memory.
Succinctly, in the time interval $[0,1/2]$, $z_{2}(t)$ is constant
and takes the value $x_{0}$, while the component $z_{1}$ is being
updated to the value  $f_{M}(x_{0})$. In the interval $[1/2,1]$, the
behaviors of $z_{1}$ and $z_{2}$ are switched: $z_{2}$ becomes
\textquotedblleft active\textquotedblright, while $z_{1}$ is
\textquotedblleft dormant\textquotedblright\ (taking value $\simeq
f_{M}(x_{0})$). In this interval, $z_{2}$ will approach the value of
$z_{1}$, i.e.\ $f_{M}(x_{0})$. Therefore, at time $t=1$,
$z_{1}(t)\simeq f_{M}(x_{0})$ and $z_{2}(t)\simeq f_{M}(x_{0})$.
Then the procedure is repeated for the interval $[1,2]$ resulting in
$z_{1}(2)\simeq f_{M}^{[2]}(x_{0})$ and $z_{2}(2)\simeq
f_{M}^{[2]}(x_{0})$. Repeating the process, one obtains
$z_{1}(j)\simeq f_{M}^{[j]}(x_{0})$ and $z_{2}(j)\simeq
f_{M}^{[j]}(x_{0})$ for all $j\in\mathbb{N}$
(cf.~Fig.~\ref{fig:branicky}, where $y_{A}=z_{1}$ and
$y_{B}=z_{2}$).

Let us look at this in more detail. First, we select the parameters
used in Construction \ref{Construction_Target} as follows:
$\gamma\leq1/4$, $t_{0}=0$, $t_{1}=1/2$, $\phi=\phi_{1}$, where
$\phi_{1}(t)=\theta_{j}(\sin2\pi t)$, and $c_{j,1}=c$ given by
(\ref{3:Target_Def_c}). Using these parameters in
(\ref{3:Branicky}), $z_{2}^{\prime}(t)=0$ for $t\in\lbrack0,1/2]$;
thus the
first equation of (\ref{3:Branicky})\ becomes%
\[
z_{1}^{\prime}=c(b-z_{1})^{3}\phi(t)
\]
where $b=f(x_0)=f_{M}(x_{0})$. It follows from Construction \ref{Construction_Target}
that $\left\vert z_{1}(1/2)-f_{M}(x_{0})\right\vert <\gamma\leq1/4$. Next, for
$t\in\lbrack1/2,1]$, $z_{1}^{\prime}(t)=0$, and Construction
\ref{Construction_r} ensures that $r(z_{1}(t))=f_{M}(x_{0})$ ($z_{1}$
\textquotedblleft remembers\textquotedblright\ the value of $f_{M}(x_{0})$ for
$t\in\lbrack1/2,1]$). If we make use of Construction \ref{Construction_Target}
again with the new set of parameters: $t_{0}=1/2$, $t_{1}=1$,
$\phi(t)=\phi_{2}(t)=\theta_{j}(-\sin2\pi t)$, and $c_{j,2}=c$ given by (\ref{3:Target_Def_c}),
then the second equation of (\ref{3:Branicky})\ becomes%
\[
z_{2}^{\prime}=c(b-z_{2})^{3}\phi(t)
\]
where $b=f_{M}(x_{0})$. Hence, one has $\left\vert z_{2}(1)-f_{M}%
(x_{0})\right\vert <\gamma\leq1/4$, which implies that $r(z_{2}(1))=f_{M}%
(x_{0})$. We now continue to the time interval $[1, 2]$. For $t\in
\lbrack1,3/2]$, $z_{2}^{\prime}(t)=0$, and Construction \ref{Construction_r}
ensures that $f(r(z_{2}(3/2)))=f_{M}^{[2]}(x_{0})$. Since both $\sin2\pi t$
and $-\sin2\pi t$ are periodic with period one, it follows that the above
procedure can be repeated for time intervals $[j,j+1]$,
$j\in\mathbb{N}$ (cf.~Fig.~\ref{fig:branicky}). Moreover, one has, for any
given $x_{0}\in\mathbb{N}$,%
\[
r(z_{2}(t))=f_{M}^{[j]}(x_{0})\text{ whenever }t\in\lbrack j,j+1/2]
\]
for all $j\in\mathbb{N}$. In this sense, (\ref{3:Branicky})
simulates the iteration of the function
$f_{M}:\mathbb{N}\rightarrow\mathbb{N}$. Since $f$ is an extension
of $f_{M}$, we have $f^{[j]}(x_{0})=f_{M}^{[j]}(x_{0})$ for any
$x_{0}\in\mathbb{N}$. For this reason, we also say that
(\ref{3:Branicky}) simulates the iteration of $f$ (over
$\mathbb{N}$). In the case where $f_{M}$ is the transition function
of a Turing machine $M$, the ODE system (\ref{3:Branicky}) is called
a simulation of the machine $M$.

The construction can be easily extended to the more general case
$h:\mathbb{N}^{k}\rightarrow\mathbb{N}^{k}$ for $k\geq1$. We then obtain an
ODE with $2k$ equations, where each component $h_{1},\ldots,h_{k}$ of $h$ is
simulated by a pair of equations.

\subsection{Simulations of Turing machines with ODEs - analytic case\label{Sec:Branickyanalytic}}

The previous section shows how to iterate a map from $\mathbb{N}$ to
$\mathbb{N}$ with smooth (but non-analytic) ODEs. In particular, we
can iterate the transition function of a given TM with some smooth
ODE.  However, the ODE built in Construction
\ref{Construction_Branicky} to iterate the function $f$ is not
analytic, even if the function $f$ itself is analytic, for the
construction uses the non-analytic functions $\theta_{j}$ in a
crucial way. Worse still, it is believed that, in general, one
cannot iterate analytic functions with analytic ODEs. However, if an
analytic function $f$ satisfies the conditions set in Theorem
\ref{Th:simulation}, in particular Condition (\ref{Eq:Th_main}),
then we show in this subsection that $f$ can be iterated by an
analytic ODE. \smallskip

The main idea underlying the construction goes as
follows. 
If we want to iterate the transition function
$f_{M}:\mathbb{N}^3\to\mathbb{N}^3$ of a Turing machine $M$ with
analytic ODEs by using a system similar to (\ref{3:Branicky}), we
cannot allow $z_{1}^{\prime}$ and $z_{2}^{\prime}$ to be $0$ in
half-unit time intervals. Instead, we allow them to be very close to
zero, which will add some errors to
system (\ref{3:Branicky}). 
In general situations, this error may cause the trajectory of the
ODE starting at some $x_{0}\in\mathbb{N}^3$ to depart from
$f_{M}^{\lbrack k\rbrack}(x_{0})$ in the long run, and thus
disqualify the ODE as an iterator of $f_{M}$. However, since $f$,
the extension of $f_{M}$, satisfies (\ref{Th:simulation}), it
simulates $M$ robustly in the presence of errors.
This allows
us to repeat the process arbitrarily many times and still maintain
$z_{1}(j)\simeq f^{[j]}(x_{0})$ for all $j\in\mathbb{N}$. Let us re-analyze
the constructions of Section \ref{Sec:Branickyinfty} from this new
perspective.\smallskip

\noindent \textbf{Convention.} From now on we fix an analytic
function $f:\mathbb{R}^3\to\mathbb{R}^3$ that satisfies Theorem
\ref{Th:simulation} and simulates a universal Turing machine $M$
(thus the restriction of $f$ on $\mathbb{N}^3$ is the transition
function $f_{M}$ of $M$). We note that the function $f$ also
satisfies Theorem \ref{Th:Main_discrete}. \smallskip

\noindent\textbf{Studying the perturbed targeting equation}%
\index{targeting equation!perturbed} (cf.~Construction
\ref{Construction_Target}). Because the iteration procedure relies
on the basic ODE (\ref{3:ODE_target}), we need to study the
following perturbed version of (\ref{3:ODE_target})
\begin{equation}
z^{\prime}=c(\overline{b}(t)-z)^{3}\phi(t) \label{3:Perturbed_target}%
\end{equation}
where $\left\vert \overline{b}(t)-b\right\vert \leq\rho$. This
\textquotedblleft perturbed\textquotedblright\ $\overline{b}(t)$
accounts for the possibility that $z_{1}(t)$ and $z_{2}(t)$ may not
be fixed in the half-unit time interval where they should be
\textquotedblleft dormant.\textquotedblright\ As in
(\ref{3:Branicky}), we take the departure time\ to be $t_{0}=0$, the
arrival time to be $t_{1}=1/2$, and $\phi
:\mathbb{R}\to\mathbb{R}_{0}^{+}$ satisfying
$\int_{0}^{1/2}\phi(t)dt>0$, where $c$ satisfies
(\ref{3:Target_Def_c}) and $\gamma>0$ is the targeting error without
perturbation, that is, $|z(1/2)-b|<\gamma$ if
$z$ is the solution to the \textquotedblleft unperturbed\textquotedblright%
\ equation $z^{\prime}=c(b-z)^{3}\phi(t)$. Let $\overline{z}$ be the solution
of this new ODE (\ref{3:Perturbed_target}) with initial condition
$\overline{z}(0)=\overline{z}_{0}$ and let $z_{+},z_{-}$ be the solutions of
$z^{\prime}=c(b+\rho-z)^{3}\phi(t)$ and $z^{\prime}=c(b-\rho-z)^{3}\phi(t)$,
respectively, with initial conditions $z_{+}(0)=z_{-}(0)=\overline{z}_{0}$.\
Since $\phi(t)$ is a nonnegative function, it is clear that, for all
$(t, z)\in\mathbb{R}^{2}$, the following holds:
\begin{equation}
c(b-\rho-z)^{3}\phi(t)\leq c(\bar{b}(t)-z)^{3}\phi(t)\leq c(b+\rho-z)^{3}%
\phi(t) \label{3:Defining_fences}%
\end{equation}
From (\ref{3:Defining_fences}) and a standard differential
inequality from the basic theory of ODEs (see e.g.~\cite[Appendix
T]{HW95}), it follows that $z_{-}(t)\leq\overline{z}(t)\leq
z_{+}(t)$ for all $t\in\mathbb{R}_{0}^{+}$. So if we have an upper
bound on $z_{+}$ and a lower bound on $z_{-}$, we immediately get
bounds for $\overline{z}$.

With $\gamma$ being the targeting error, we have the following estimates from
Construction \ref{Construction_Target}:
\[
|b+\rho-z_{+}(1/2)|<\gamma\ \text{and}\ |b-\rho-z_{-}(1/2)|<\gamma
\]
which in turn implies that
\[
b-\rho-\gamma<z_{-}(1/2)\leq\bar{z}(1/2)\leq z_{+}(1/2)<b+\rho+\gamma
\]
or equivalently%
\begin{equation}
\left\vert \overline{z}(1/2)-b\right\vert <\rho+\gamma
\label{4:Error_target_perturbed}%
\end{equation}
In other words, the targeting error in the presence of perturbation
is at most $\rho+\gamma$, if the target value is perturbed at most
up to an amount $\rho$. \medskip

\noindent\textbf{Removing }$\theta_{j}$ \textbf{ from (\ref{3:Branicky}%
).}%
\index{function!thetaj@$\theta_{j}$}
We must remove the function $\theta_{j}$ from the right-hand side of (\ref{3:Branicky})
as well as in the function $r$ (see Construction \ref{Construction_r}).
Since $f$ is robust to perturbations in the sense that $f$ satisfies
Condition (\ref{Eq:Th_main}) of Theorem \ref{Th:simulation}, we no
longer need the corrections performed by $r$. (For simplicity we
assume that $f:\mathbb{R}\to\mathbb{R}$ instead of
$f:\mathbb{R}^3\to\mathbb{R}^3$.)
On the other hand, we cannot simply drop the functions
$\theta_{j}(\pm\sin2\pi t)$ in (\ref{3:Branicky}). We need to
replace $\phi(t)=\theta_{j}(\sin2\pi t)$ by an analytic function
$\zeta:\mathbb{R}\rightarrow\mathbb{R}$ with the following ideal
behavior:\smallskip

(i) $\zeta$ has period $1$;

(ii) $\zeta(t)=0$ for $t\in\lbrack1/2,1]$;

(iii) $\zeta(t)\geq0$ for $t\in\lbrack0,1/2]$ \ and $\int_{0}^{1/2}%
\zeta(t)dt>0$.\smallskip

Of course, Conditions (ii) and (iii) are incompatible for analytic
functions (it is well-known that if a real analytic function is 0 in
a non-empty open interval, then it must be 0 everywhere). Instead,
we approach $\zeta$ using a function $\zeta_{\epsilon}$,
$\epsilon>0$. This function should satisfy the following
conditions:\smallskip

(ii${}^{\prime}$) $\left\vert \zeta_{\epsilon}(t)\right\vert \leq\epsilon$ for
$t\in\lbrack1/2,1]$;

(iii${}^{\prime}$) $\zeta_{\epsilon}(t)\geq0$ for $t\in\lbrack0,1/2]$ \ and
$\int_{0}^{1/2}\zeta_{\epsilon}(t)dt>I>0$, where $I$ is independent of
$\epsilon$.\smallskip

Our idea to define such a function $\zeta_{\epsilon}$ is to make use of the
function $l_{2}$ introduced in Proposition \ref{Prop:l2}: Let%
\begin{equation}
\zeta_{\epsilon}(t)=l_{2}(\vartheta(t),1/\epsilon) \label{3:Zeta_Epsilon}%
\end{equation}
where $\epsilon>0$ is the precision up to which $\zeta_{\epsilon}$ should
approximate $0$ in the interval $[1/2,1]$ and $\vartheta:\mathbb{R}%
\rightarrow\mathbb{R}$ is an analytic periodic function of period $1$
satisfying the following conditions:\smallskip

(a) $\left\vert \vartheta(t)\right\vert \leq1/4$ for $t\in\lbrack1/2,1]$;

(b) $\vartheta(t)\geq3/4$ for $t\in(a,b)\subseteq(0,1/2)$, $a<b$.\smallskip

We note that Proposition \ref{Prop:l2} and Condition (a) ensure that
$\left\vert \zeta_{\epsilon}(t)\right\vert <\epsilon$ for
$t\in\lbrack1/2,1]$, and thus ensures (ii${}^{\prime}$), while
Proposition \ref{Prop:l2} and Condition (b) guarantee that
$\left\vert \zeta_{\epsilon}(t)\right\vert >1-\epsilon$ for
$t\in(a,b)$, which in turn implies
$\int_{0}^{1/2}\zeta_{\epsilon}(t)\geq(1-\epsilon )(b-a)>3(b-a)/4$
for $\epsilon<1/4$, thus (iii${}^{\prime}$) is satisfied. We note
that, for all $(t,x)\in\mathbb{R}^{2}$, $l_{2}(t,x)>0$ and thus
$\zeta_{\epsilon}(t)\geq0$ for all $t\in\mathbb{R}$. It is not
difficult to see that one may select
$\vartheta:\mathbb{R}\rightarrow\mathbb{R}$ as %
\begin{equation}
\vartheta(t)=\frac{1}{2}(\sin^{2}(2\pi t)+\sin(2\pi t))
\label{3:Periodic_function_for_l2}%
\end{equation}
since this function satisfies both Conditions (a) and (b)
(e.g.~$a=0.16$ and $b=0.34$). Now we make the replacement: replace
$\theta_{j}(\sin2\pi t)$ by the analytic function
$\zeta_{\epsilon}(t)=l_{2}(\vartheta(t),1/\epsilon)$, where
$\vartheta$ is given by (\ref{3:Periodic_function_for_l2}).
Similarly, we replace $\theta_{j}(-\sin2\pi t)$ by the analytic
function $\zeta _{\epsilon}(-t)$. For this particular function
$\vartheta$ defined by (\ref{3:Periodic_function_for_l2}), the
constant satisfying (\ref{3:Target_Def_c}) may be selected as
\[
c\geq\frac{1}{2\gamma^{2}\frac{3(0.34-0.16)}{4}}%
\]
which is independent of $\epsilon$.\medskip

\noindent\textbf{Performing Construction
\ref{Construction_Branicky}\ with analytic functions.} We are now
ready to perform a simulation of the Turing machine $M$  (or an iteration of the transition
function $f_{M}:\mathbb{N}^3\to\mathbb{N}^3$ of $M$)
with a system similar to (\ref{3:Branicky}), but using only analytic
functions. For readability, let us assume for now that
$f_{M}:\mathbb{N}\to\mathbb{N}$ and $f:\mathbb{R}\to\mathbb{R}$
(instead of $f_{M}:\mathbb{N}^3\to\mathbb{N}^3$ and $f:\mathbb{R}^3\to\mathbb{R}^3$). 
Choose a targeting error $\gamma>0$ and suppose
$\varepsilon>0$ is the error resulting from perturbation such that
\begin{equation}
2\gamma\leq\varepsilon<1/8, \label{3:Targeting_error_perturbed_Branicky}%
\end{equation}
and consider the following system of ODEs
\begin{equation}
\left\{
\begin{array}
[c]{l}%
z_{1}^{\prime}=c_{1}(f\circ\sigma^{\lbrack k]}(z_{2})-z_{1})^{3}%
\,\zeta_{\epsilon_{1}}(t)\\
z_{2}^{\prime}=c_{2}(\sigma^{\lbrack n]}(z_{1})-z_{2})^{3}\,\zeta
_{\epsilon_{2}}(-t)
\end{array}
\right.  \label{3:Branicky_analytic}%
\end{equation}
with initial conditions $z_{1}(0)=z_{2}(0)=\overline{x}_{0}$, where
$\left\vert x_{0}-\overline{x}_{0}\right\vert \leq\varepsilon$,
$x_{0}\in\mathbb{N}$, and $\sigma$ is the error-contracting function
defined in (\ref{3:sigma}). The four constants $c_{1}, c_{2}, k$,
and $n$ and the two functions $\epsilon_{1}$ and $\epsilon_{2}$
remain to be defined.

We would like for (\ref{3:Branicky_analytic}) to satisfy the
following property: on $[0,1/2]$,
\begin{equation}
\left\vert z_{2}^{\prime}(t)\right\vert \leq\gamma\label{3:Bound_z_2}%
\end{equation}
This can be achieved by setting $\epsilon_{2}(t)=\gamma/K(t)$, where
$K(t)=c_{2}^{4/3}(\sigma^{[n]}(z_{1})-z_{2})^{4}+1$. We note that $|x|^{3}\leq
x^{4}+1$ for all $x\in\mathbb{R}$ and $\vartheta(-t)<\frac{1}{4}$ for all
$t\in[n, n+\frac{1}{2}]$, $n\in\mathbb{N}$. Thus, from Proposition
\ref{Prop:l2},
\[
0\leq\zeta_{\epsilon_{2}}(-t)=\zeta_{\epsilon_{2}(t)}(-t)=l_{2}(\vartheta(-t),
1/\epsilon_{2}(t))\leq\epsilon_{2}(t), \, t\in[n, n+\frac{1}{2}], \,
n\in\mathbb{N}%
\]
which further implies that, for all $n\in\mathbb{N}$ and $t\in[n,n+\frac{1}%
{2}]$,
\begin{align*}
|z^{\prime}_{2}(t)|  &  =  |c_{2}(\sigma^{[n]}(z_{1})-z_{2})^{3}%
\zeta_{\epsilon_{2}}(-t)|\\
&  \leq |c_{2}(\sigma^{[n]}(z_{1})-z_{2})^{3}|\cdot|\epsilon_{2}(t)|\\
&  =  |c_{2}(\sigma^{[n]}(z_{1})-z_{2})^{3}|\frac{\gamma}{c_{2}^{4/3}%
(\sigma^{[n]}(z_{1})-z_{2})^{4}+1}\\
&  <  \gamma
\end{align*}
Similarly, if we set
$\epsilon_{1}(t)=\gamma/[c_{1}^{4/3}(f\circ\sigma
^{[n]}(z_{2})-z_{1})^{4}+1]$, then we have
$|z_{1}^{\prime}(t)|<\gamma$ for all $t\in[n+\frac{1}{2}, n+1]$,
$n\in\mathbb{N}$. We still need to define the two constants $c_{1}$
and $c_{2}$ which must satisfy (\ref{3:Target_Def_c}). To do so, we
first observe that $\epsilon_{2}(t)\leq\gamma<\frac{1}{16}$ (see
(\ref{3:Targeting_error_perturbed_Branicky})) for $t\in[n,
n+\frac{1}{2}]$, $n\in\mathbb{N}$. Then it follows from
(\ref{3:Periodic_function_for_l2}) and the discussions immediately
preceding and following
(\ref{3:Periodic_function_for_l2}) that%
\[
\int_{0}^{1/2}\zeta_{\epsilon_{2}(t)}(-t)dt>(1-\epsilon_{2}%
(t))(0.34-0.16)>(1-\frac{1}{16})(0.34-0.16)>0.16
\]
Thus if we choose $c_{2}$ such that
\[
c_{2}\geq\frac{1}{2\gamma^{2}(0.16)}
\]
then $c_{2}$ satisfies (\ref{3:Target_Def_c}). We choose $c_{1}$ in
a similar way. The two integers $k$ and $n$ are chosen such that
$k=n$ and $|\sigma^{[k]}(\epsilon+\frac{\gamma}{2})|<\gamma$.

Next we proceed to show that system (\ref{3:Branicky_analytic})
iterates the function $f$. First we note that, from the given
initial condition $|z_{2}(0)-x_{0}|<\epsilon$ with
$x_{0}\in\mathbb{N}$, it follows from
(\ref{3:Bound_z_2}) that for all $t\in[0, 1/2]$,%
\[
|z_{2}(t)-x_{0}|\leq|z_{2}(t)-z_{2}(0)|+|z_{2}(0)-x_{0}|<\gamma t+\epsilon
\leq\frac{\gamma}{2}+\epsilon
\]
Then from our choice of $k$, $|\sigma^{[k]}(z_{2}(t))-x_{0}|<\gamma$ for all
$t\in[0, 1/2]$, which further implies that $|f\circ\sigma^{[k]}(z_{2}%
(t))-f(x_{0})|<\lambda\gamma<\gamma$, $t\in[0, 1/2]$ (see Theorem
\ref{Th:simulation}). Now from the study of the perturbed targeting equation
(\ref{3:Perturbed_target}), if we let $\phi(t)=\zeta_{\epsilon_{1}}(t)$ and
the perturbation error $\rho=\gamma$ in (\ref{3:Perturbed_target}), then we
reach the conclusion that the solution $z_{1}$ exists on $[0, 1/2]$ and
\begin{equation}
\label{3:z1_after_time_1/2}|z_{1}(1/2)-f(x_{0})|<2\gamma\leq\epsilon
\end{equation}
Proceeding to the next half-time interval $[1/2, 1]$, the roles of
$z_{1}$ and $z_{2}$ are switched. On $[1/2, 1]$,
$|z^{\prime}_{1}(t)|\leq\gamma$. Then it follows from
$|z^{\prime}_{1}(t)|\leq\gamma$ and (\ref{3:z1_after_time_1/2}) that
\[
\mbox{$|z_{1}(t)-f(x_{0})|\leq \epsilon +\gamma/2\leq 1/4$ for
all  $t\in [1/2, 1]$}
\]
Thus from the choice of $n$, $|\sigma^{[n]}(z_{1}(t))-f(x_{0})|<\gamma$ for
all $t\in[1/2, 1]$. Again, by making use of the perturbed targeting equation
(\ref{3:Perturbed_target}) with $\phi(t)=\zeta_{\epsilon_{1}}(t)$, we obtain
\[
|z_{2}(1)-f(x_{0})|<2\gamma\leq\epsilon
\]

Repeating the above procedure for intervals $[1,2]$, $[2,3]$, etc., we
conclude that for all $j\in\mathbb{N}$ and for all $t\in\lbrack j,j+1/2]$,
\begin{equation}
\left\vert z_{1}(t)-f^{[j]}(x_{0})\right\vert \leq\epsilon+\frac{\gamma}%
{2}\leq1/4 \label{Eq:eta}%
\end{equation}
Moreover, $z_{1}$ is defined as the solution of an analytic ODE.

From the construction above, it is easy to see that the following
ODE system (\ref{3:Branicky_analytic_full}) iterates
$f_{M}:\mathbb{N}^3\to\mathbb{N}^3$:
\begin{equation}
\left\{
\begin{array}
[c]{l}%
y_{1}^{\prime}=c_{1}(f_{1}(\sigma^{\lbrack k]}(v_{1}), \sigma^{\lbrack k]}(v_{2}), \sigma^{\lbrack k]}(v_{3}))-y_{1})^{3}%
\,\zeta_{\epsilon_{1}}(t)\\
y_{2}^{\prime}=c_{1}(f_{2}(\sigma^{\lbrack k]}(v_{1}), \sigma^{\lbrack k]}(v_{2}), \sigma^{\lbrack k]}(v_{3}))-y_{2})^{3}%
\,\zeta_{\epsilon_{1}}(t)\\
y_{3}^{\prime}=c_{1}(f_{3}(\sigma^{\lbrack k]}(v_{1}), \sigma^{\lbrack k]}(v_{2}), \sigma^{\lbrack k]}(v_{3}))-y_{3})^{3}%
\,\zeta_{\epsilon_{1}}(t)\\
v_{1}^{\prime}=c_{2}(\sigma^{\lbrack n]}(y_{1})-v_{1})^{3}\,\zeta
_{\epsilon_{2}}(-t)  \\
v_{2}^{\prime}=c_{2}(\sigma^{\lbrack n]}(y_{2})-v_{2})^{3}\,\zeta
_{\epsilon_{2}}(-t) \\
v_{3}^{\prime}=c_{2}(\sigma^{\lbrack n]}(y_{3})-v_{3})^{3}\,\zeta
_{\epsilon_{2}}(-t)
\end{array}
\right.  \label{3:Branicky_analytic_full}%
\end{equation}
where $(y_1, y_2, y_3), (v_1, v_2, v_3)\in\mathbb{R}^3$ and
$f=(f_{1}, f_{2},f_{3})$, $f_i:\mathbb{R}^3\to\mathbb{R}$, $1\leq i\leq 3$.

We need to build several new systems similar to
(\ref{3:Branicky_analytic_full}) but with $\zeta _{\epsilon_{1}}$
and $\zeta _{\epsilon_{2}}$ replaced by other functions.  For
simplicity, we will not present their fully expanded forms as in
(\ref{3:Branicky_analytic_full}), instead we will use systems
similar to (\ref{3:Branicky_analytic}) to represent the
corresponding fully expanded systems.  
When doing so, the expression $(f\circ \sigma)(x)\phi(t)$ is used
for $f(\sigma(x_1)\phi(t), \sigma(x_2)\phi(t), \sigma(x_3)\phi(t))$
if $f:\mathbb{R}^3\to\mathbb{R}^3$, $\sigma,
\phi:\mathbb{R}\to\mathbb{R}$, $x\in\mathbb{R}^3$, and
$t\in\mathbb{R}$. 

\subsection{The halting configuration is a sink\label{Sec:Hyper}}

In the previous subsection we have shown how to iterate the map $f$
with an analytic ODE on $\mathbb{R}^{6}$. Since $f$ simulates the
Turing machine $M$, so does the ODE. Let us continue to assume that
the machine $M$ has a unique halting configuration. To prove Theorem
\ref{Th:Main_continuous},
we need to show that this halting configuration of $M$ corresponds
to a hyperbolic sink of the analytic ODE simulating $M$.

Note that the results of Section \ref{Sec:Branickyanalytic} do not guarantee
the existence of such a hyperbolic sink. What is shown there is that if
$x_{0}\in\mathbb{N}^{3}$ encodes an initial configuration of $M$ and $x_{h}%
\in\mathbb{N}^{3}$ corresponds to the halting configuration of $M$, then $M$ halts on $x_{0}$ iff%
\[
\left\Vert z(0)-(x_{0},x_{0})\right\Vert \leq1/8\text{ \ \
}\Longrightarrow \text{ \ \ }\exists k_{0}\in\mathbb{N} \text{ such
that } \forall t\geq k_{0},\text{ }\left\Vert
z(t)-(x_{h},x_{h})\right\Vert \leq1/4
\]
where $z$ is the solution of (\ref{3:Branicky_analytic}) with the
initial condition $z(0)\in\mathbb{R}^6$. In other words, we know
that, provided that $M$ halts on $x_{0}$, a trajectory starting
sufficiently close to $(x_{0},x_{0})$ will eventually be in the ball
$B((x_{h},x_{h}),1/4)$, but it may just wander there and never
converge to $(x_{h},x_{h})$. Since we need to make $(x_{h},x_{h})$ a
sink (hyperbolicity will be dealt with in the next subsection), we
have to modify system (\ref{3:Branicky_analytic}).
There are several issues to be addressed.

$\bullet$ In general, an ODE (\ref{3:ODE_f}) describes a dynamical system if it is
autonomous, i.e.\ if $g$ in (\ref{3:ODE_f}) does not depend on $t$ (see
e.g.\ \cite{HS74}), which is not the case for system (\ref{3:Branicky_analytic}%
). It is true that (\ref{3:ODE_f}) can be converted into an autonomous system
by writing%
\[
\left\{
\begin{array}
[c]{l}%
y^{\prime}=g(z,y)\\
z^{\prime}=1
\end{array}
\right.
\]
But in a system like this, if it has a hyperbolic sink $\alpha$, the
\textquotedblleft time\textquotedblright\ variable $z$ must be
finite at $\alpha$. Such a system cannot simulate a universal Turing
machine, for the set of halting times of a universal Turing machine
is not bounded. To deal with this problem, an intuitive fix is to
introduce a new variable $u=e^{-t}$. Then $u^{\prime}=-u$ has an
equilibrium point at $u=0$ and $u$ converges exponentially fast to
$0$ as $t\rightarrow\infty$. The analytic ODE
(\ref{3:Branicky_analytic}) can then be rewritten as follows:
\begin{equation}
\left\{
\begin{array}
[c]{l}%
z_{1}^{\prime}=c_{1}(f\circ\sigma^{\lbrack k]}(z_{2})-z_{1})^{3}%
\zeta_{\epsilon_{1}}(-\ln u)\\
z_{2}^{\prime}=c_{2}(\sigma^{\lbrack n]}(z_{1})-z_{2})^{3}\,\zeta
_{\epsilon_{2}}(\ln u)\\
u^{\prime}=-u
\end{array}
\right.  \,\label{Eq:system1}%
\end{equation}
It is clear that the new system (\ref{Eq:system1}) still simulates
the machine $M$ as system (\ref{3:Branicky_analytic}) does. However,
a new problem occurs with the introduction of $u$: system
(\ref{Eq:system1}) has no fixed point, let alone a sink. The problem
is obvious: the system is not defined at the only possible
equilibrium point $(x_{h},x_{h},0)$ since $\zeta_{\epsilon_{1}}(-\ln
u)$ and $\zeta_{\epsilon_{2}}(\ln u)$ are not defined at $u=0$.
Moreover, the problem cannot be easily solved by extending
$\zeta_{\epsilon_{1}}(-\ln u)$ and $\zeta_{\epsilon_{2}}(\ln u)$ as
$u\rightarrow0^{+}$, for both $\zeta_{\epsilon_{1}}$ and $\zeta_{\epsilon_{2}%
}$ are periodic functions. The strategy we use to fix this problem
is to replace $\zeta_{\epsilon_{1}}(-\ln u)$ and
$\zeta_{\epsilon_{2}}(\ln u)$ by other functions such that the
resulting system not only still simulates the machine $M$ but is
also analytic; in particular, it is analytic at $(x_{h},x_{h},0)$.
More precisely, our idea is to substitute $\zeta_{\epsilon_{1}}(-\ln
u)$ ($\zeta_{\epsilon_{2}}(\ln u)$ as well) by a function $\chi(u)$
that acts distinctively depending upon whether or not the machine
$M$ has halted: as long as $M$ has not yet halted, $\chi(u)$ behaves
roughly like $\zeta_{\epsilon_{1}}(-\ln u)$ (this ensures that the
modified system with $\chi(u)$ replacing $\zeta_{\epsilon_{1}}(-\ln
u)$ still simulates the machine $M$); once $M$ has halted, the
behavior of $\chi(u)$ is switched so that $\chi(u)$ is defined as
well as analytic at $u=0$ and the modified system converges to
$(x_{h},x_{h},0)$ hyperbolically. The switch of behavior of
$\chi(u)$ is controlled by a variable in (\ref{Eq:system1}) that
codes the current state of $M$ in the simulation. For example, one
can use the third component $v_{3}$ of $z_{2}\in\mathbb{R}^{3}$,
which codes the state of the machine $M$. If the states are given by
the values $1,\ldots,m$, with $m$ corresponding to the halting
state, then the condition $v_{3}(t)\leq m-3/4$ implies that the
machine $M$ hasn't yet halted, while the condition $v_{3}(t)\geq
m-1/4$ implies that $M$ has halted (the error of 1/4 here is due to
the perturbation error allowed in the simulation). We note that if
$M$ never halts on an input $(x_{0},x_{0})$, then system
(\ref{Eq:system1}) with the initial value $(x_{0},x_{0},1)$ will not
converge to $(x_{h},x_{h},0)$. In this case, it does not matter
whether or not the system is defined at $u=0$ (i.e.
$t=\infty$). %

$\bullet$ So we want a function $\chi(u)$ to replace
$\zeta_{\epsilon_{1}}(-\ln u)$ that has two distinct behaviors
depending upon the value of the variable $v_{3}$. For this purpose,
we add a new equation with a new variable $\tau$ to system
(\ref{Eq:system1}) which is defined by%
\begin{equation} \label{Eq:Def_s}
\tau^{\prime}=\frac{d\tau}{dt}=-\alpha u^{\alpha+1}
\end{equation}
where ideally $\alpha=-1$ before the machine $M$ halts, and
$\alpha=1$ after $M$ halts. Our aim is to make the point $(x_{h},
x_{h}, 0, 0)$ a hyperbolic sink of system (\ref{Eq:system1})
augmented with equation (\ref{Eq:Def_s}), where $x_h$ is the halting
configuration of $M$.  We observe that before $M$ halts,
$\tau^{\prime}=1$, and so $\tau(t)=t+\tau(0)$ (thus $\tau(u)=-\ln
u+\tau(0)$). After $M$ halts, $\tau^{\prime}=-u^2$; thus
$\tau$ is defined and analytic at $u=0$. Thus if we define $\chi(u)$ as follows:%
\[
\chi(u)=\zeta_{\epsilon_{1}}(\tau(u))
\]
then this $\chi(u)$ will meet our requirement of changing behavior
according to whether or not the machine $M$ has halted.
Unfortunately, there is a shortfall: defined in such a way,
$\chi(u)$ is not analytic due to the jump in $\alpha$. For this
reason, we modify the definition of $\alpha$ such that
$\alpha\simeq-1$ before the machine $M$ halts, $\alpha\simeq1$ after
$M$ halts, and $\alpha$ is analytic in $v_{3}$. Precisely, we define
\[
\alpha=\frac{5}{2}\xi_{2}(v_{3}-(m-1),10)-\frac{5}{4}%
\]
(the 10 is rather arbitrary) where $v_{3}$ is the third component of $z_{2}%
\in\mathbb{R}^{3}$ that codes the state of the machine $M$. Since
$\xi_{2}(\cdot, 10)$ is analytic, so is $\alpha$. It is not
difficult to show (see \cite{GCB08}) that for any $x<\frac{1}{4}$
and $y>0$, if $|l_{2}(x,y)|<\frac{1}{y}$, then $|\xi
_{2}(x,y)|<\frac{1}{y}$. But, as a consequence of Proposition
\ref{Prop:l2}, one can easily derive that $|l_{2}(x,y)|<\frac{1}{y}$
for any $x<\frac{1}{4}$ and $y>0$, which in turn implies that
\begin{equation}
\mbox{$|\xi_{2}(x,y)|<\frac{1}{y}$ for any $x<\frac{1}{4}$ and
$y>0$}\label{xi_two_bound}%
\end{equation}
Since we assume that states are coded by integers $1,\ldots,m$ with
$m$ the halting state, it follows that, until $M$ halts (which may
never happen), $0<\xi_{2}(v_{3}-(m-1),10)\leq1/10$. One cycle after
$M$ has halted (this ensures $v_{3}$ has enough time to update its
value to $\simeq m$), one
has $9/10<\xi_{2}(v_{3}-(m-1),10)\leq1$. Thus%
\begin{equation}
\alpha\in\left\{
\begin{array}
[c]{ll}%
\left(  -\frac{5}{4},-1\text{ }\right]  \text{\ \ \ } & \text{if }M\text{
hasn't halted}\\
\left[  1,\frac{5}{4}\right)   & \text{one cycle after }M\text{ has halted}%
\end{array}
\right.  \label{Eq:Def_alpha}%
\end{equation}

$\bullet$ If we simply define $\tau$ by (\ref{Eq:Def_s}) with
analytic $\alpha$, we will not be able to show
that system (\ref{Eq:system1}) augmented with Equation (\ref{Eq:Def_s}%
) has $(x_{h}, x_{h}, 0, 0)$ as a hyperbolic sink. This is because
we need to compute the Jacobian of the augmented system at
$(x_{h},x_{h},0,0)$, which includes the derivative of $u^{\alpha
+1}$ with respect to $v_{3}$, since this variable appears in the
expression for $\alpha$. But $\partial{u^{\alpha
+1}}/\partial{v_{3}}=u^{\alpha +1}\ln u\frac{\partial\alpha
}{\partial v_{3}}$, which is not defined at $u=0$. For this reason,
we have to modify the
definition of $\tau$ yet again, using the following equation%
\begin{equation}
\tau^{\prime}=2(u+\xi_{2}(v_{3}-(m-1),\tau+1))^{\alpha+1}\xi_{2}(m-v_{3}%
,10)-\xi_{2}(v_{3}-(m-1),10+10\tau^{2})\tau\label{Eq:u_2_derivative}%
\end{equation}
We note that the right-hand side of (\ref{Eq:u_2_derivative}) is a
function of $u$ and $\tau$. Since the derivative of this function
with respect to $v_3$ containing $\ln (u+1)$ (see the calculation
below) and the function $\xi_{2}(x, y)$ is only defined for $y>0$,
the function in the right-hand side of (\ref{Eq:u_2_derivative}) is
defined only for $u>-1$ and $\tau>-1$.

The behavior of $\tau$, defined by (\ref{Eq:u_2_derivative}), is
remarkably similar to its previous form defined in (\ref{Eq:Def_s})
as we shall see below. The parcel
$(u+\xi_{2}(v_{3}-(m-1),\tau+1))^{\alpha+1}$ in the product of the
right-hand side of (\ref{Eq:u_2_derivative}) ensures that, in the
halting configuration, the derivative of
$(u+\xi_{2}(v_{3}-(m-1),\tau+1))^{\alpha}$  with respect to $v_{3}$
is well defined, and thus fixes the problem raised above (1 is added
to $\tau$ for the reason that $\xi_{2}(x,y)$ is defined only for
$y>0$ but $\tau=0$
at $(x_{h}, x_{h}, 0, 0)$). We recall that $\xi_{2}$ is an analytic function on $\mathbb{R}%
\times\mathbb{R}^{+}$, thus $\frac{\partial\xi_{2}}{\partial v_{3}}$
and $\frac{\partial\alpha}{\partial v_{3}}$ are well-defined, in
particular at the halting configuration (where $v_3=m$). For the
parcel $(u+\xi_{2}(v_{3}-(m-1),\tau +1))^{\alpha}$, we have
\begin{align*}
&   \frac{\partial}{\partial v_{3}}\left[  (u+\xi_{2}(v_{3}-(m-1),\tau
+1))^{\alpha}\right] \\
&  =  (u+\xi_{2}(v_{3}-(m-1),\tau+1))^{\alpha}\ln(u+\xi_{2}(v_{3}%
-(m-1),\tau+1))\cdot\frac{\partial\alpha}{\partial v_{3}}\\
&   + \alpha(u+\xi_{2}(v_{3}-(m-1),\tau+1))^{\alpha-1}\frac{\partial}{\partial
v_{3}}(u+\xi_{2}(v_{3}-(m-1),\tau+1))
\end{align*}
which is also well-defined at the halting configuration, where
$u+\xi _{2}(v_{3}-(m-1),\tau+1)=u+\xi_{2}(m-(m-1),
\tau+1)=u+\xi_{2}(1, \tau+1)=u+1$ and thus
$\ln(u+\xi_{2}(v_{3}-(m-1),\tau+1))=\ln(u+1)$ is well-defined at
$u=0$. The other term ensures the \textquotedblleft switching of
behavior\textquotedblright\ according to whether or not $M$ has
halted.

$\bullet$ We now have the following modified system \\
\begin{equation}
\left\{
\begin{array}
[c]{l}%
z_{1}^{\prime}=c_{1}(f\circ\sigma^{\lbrack k]}(z_{2})-z_{1})^{3}%
\zeta_{\epsilon_{1}}(\tau)\\
z_{2}^{\prime}=c_{2}(\sigma^{\lbrack n]}(z_{1})-z_{2})^{3}\,\zeta
_{\epsilon_{2}}(-\tau)\\
u^{\prime}=-u\\
\tau^{\prime}=2(u+\xi_{2}(v_{3}-(m-1),\tau+1))^{\alpha+1}\xi_{2}(m-v_{3}%
,10)-\xi_{2}(v_{3}-(m-1),10+10\tau^{2})\tau
\end{array}
\right.  \label{Eq:New_TM2}%
\end{equation}
We note that the modified system is defined on $\mathbb{R}^6\times
(-1,+\infty)\times (-1,+\infty)$.

In the remainder of this subsection we prove that (I)
$z_{equilibrium}=(z_{f},z_{f},0,0)$ is an equilibrium point of this
new system, where \[z_{f}=(0,0,m)\]  is the halting configuration of
the machine $M$ (i.e. $z_{f}=x_{h}$); (II) For any trajectory
$(z_{1}(t), z_{2}(t), u(t), \tau(t))$ of (\ref{Eq:New_TM2}), if
there exists a $t_{0}>0$ such that $(z_{1}(t_0), z_{2}(t_0), u(t_0),
\tau(t_0))\in U_{f}$, where
\[ U_{f}=B(z_{f}, 1/8)\times B(z_{f}, 1/8)\times (-1, 2)\times (-1,
5)\]
is an open neighborhood of the equilibrium point $(z_{f},
z_{f}, 0, 0)$, then $(z_{1}(t), z_{2}(t), u(t), \tau(t))\in U_{f}$
for all $t\geq t_0$ and $(z_{1}(t), z_{2}(t), u(t),
\tau(t))\rightarrow (z_{f}, z_{f}, 0, 0)$ as $t\rightarrow \infty$.
Combining (I) and (II) we conclude that $(z_{f}, z_{f}, 0, 0)$ is a
sink of system (\ref{Eq:New_TM2}).


To show that (I) holds, that is, $(z_{f},z_{f},0,0)$ is an
equilibrium point of system (\ref{Eq:New_TM2}), it suffices to
observe that $f\circ\sigma^{\lbrack k]}(z_{f})=f(\sigma^{\lbrack
k]}(0), \sigma^{\lbrack k]}(0), \sigma^{\lbrack k]}(m))=z_{f}$
(according to our convention stated right before subsection
\ref{Sec:Hyper}) and $\xi_{2}(m-v_{3},10)=\xi_{2}(0, 10)=0$ at
$z_{f}=(0,0,m)$. The first equality follows from (\ref{3:sigma}) and
the fact that $f(z_f)=f_{M}(z_f)=z_f$ (since $f$ extends $f_{M}$,
$f_{M}$ is the transition function of $M$, and $z_{f}=(0, 0, m)$ is
the halting configuration of $M$). The second equality is derived
from (\ref{Eq:Eta2}). Thus (I) holds.

We now consider (II). We need to show that  if $(z_{1}(t_0),
z_{2}(t_0), u(t_0), \tau(t_0))\in U_{f}$ for some $t_0>0$, then
$(z_{1}(t), z_{2}(t), u(t), \tau(t))\in U_{f}$ for all $t\geq t_0$
and $(z_{1}(t), z_{2}(t), u(t), \tau(t))\rightarrow (z_{f}, z_{f},
0, 0)$ as $t\rightarrow \infty$. Since $u^{\prime}=-u$,
$u(t)=u(0)e^{-t}$. It is clear that $|u(t)|\leq |u(0)|$ for any
$t\geq 0$ and $u(t)\rightarrow 0$ as $t\rightarrow\infty$ for any
$u(0)\in (-1, 2)$ (thus any $t>0$ can be used as $t_{0}$). However,
to prove the rest of (II),  the functions $\zeta_{\epsilon_{1}}$ and
$\zeta_{\epsilon_{2}}$ in (\ref{Eq:New_TM2}) must be nonzero if
$|v_3-m|<1/4$ for technical reasons. Thus we need to refine the two
functions $\zeta_{\epsilon_{1}}$ and $\zeta_{\epsilon_{2}}$ yet
again. Recall that
$\zeta_{\epsilon}(t)=l_{2}(\vartheta(t),\frac{1}{\epsilon})$, where
$\vartheta(t)=\frac{1}{2}(\sin^{2}(2\pi t)+\sin(2\pi t))$ (see
(\ref{3:Zeta_Epsilon}) and (\ref{3:Periodic_function_for_l2})). We
redefine $\zeta_{\epsilon_{1}}$ and $\zeta_{\epsilon_{2}}$ by
choosing a different function for $\vartheta(t)$; the new
$\vartheta(t)$ is defined as follows:
\[
\vartheta(t)=\frac{1}{2}(\sin^{2}(2\pi t)+\sin(2\pi t))+l_{2}(v_{3}%
-(m-1),t+10)
\]
where, as already mentioned, $v_{3}$ is the third component of
$z_{2}\in\mathbb{R}^{3}$ that codes the state of the machine $M$.
(Again the number $10$ is rather arbitrary.) Recall that the states
of $M$ are coded by integers $1,2,\ldots,m$ with $m$ being the
halting state. For any $v_3$ satisfying $|v_3-m|<1/4$, we have
$|v_3-(m-1)-1|<1/4$, thus  it follows that
$|l_{2}(v_{3}-(m-1),t+10)-1|<\frac{1}{t+10}$ (see Proposition
\ref{Prop:l2}). Therefore, $\vartheta
(t)\geq-\frac{1}{8}+1-\frac{1}{t+10}=\frac{7}{8}-\frac{1}{t+10}\geq\frac{3}{4}$
(note that the minimum of $\frac{1}{2}(\sin^{2}(2\pi t)+\sin(2\pi
t))$ is $-\frac{1}{8}$), which implies that
$1-\vartheta(t)<\frac{1}{4}$, and it follows again from Proposition
\ref{Prop:l2} that  $|1-\zeta
_{\epsilon}|=|1-l_{2}(\vartheta(t),\frac{1}{\epsilon})|\leq\epsilon$.

Thus, assuming that $0<\epsilon_{1},\epsilon_{2}\leq1/10$ in \
$\zeta _{\epsilon_{1}}$ and $\zeta_{\epsilon_{1}}$ (if this is not
the case, for the \textquotedblleft old\textquotedblright\ precision
$\epsilon_{i}$, substitute the improved accuracy
$\epsilon_{i}/(10\epsilon_{i}+10)$, then this new precision will be
less than $\epsilon_{i}$ and also less than 1/10), we conclude that
$\zeta_{\epsilon_{1}}$ and $\zeta_{\epsilon_{2}}$, defined in
(\ref{3:Zeta_Epsilon}) but using the new function $\vartheta$ (and
perhaps a new precision $\epsilon_{1}$ or $\epsilon_{2}$), will
satisfy
\[
9/10\leq\zeta_{\epsilon_{1}}(\tau),\zeta_{\epsilon_{2}}(\tau)<1
\]
In particular%
\begin{equation}
\frac{9}{10}\leq\left.  \zeta_{\epsilon_{1}}(\tau)\right\vert
_{z_{equilibrium}},\left.  \zeta_{\epsilon_2}(\tau)\right\vert
_{z_{equilibrium}}<1\label{Eq:zetas}%
\end{equation}

We are now ready to prove the rest of (II).  Let $(z_{1}(t),
z_{2}(t), u(t), \tau(t))$ be a trajectory of (\ref{Eq:New_TM2}) such
that  $z_{1}(t),z_{2}(t)\in B(z_{f},1/8)$ for some $t_0>0$. Then
from the subsection Performing Construction
\ref{Construction_Branicky} with Analytic Functions above, it
follows that $z_{1}(t),z_{2}(t)\in B(z_{f},1/4)$ for all $t\geq
t_0$. In the following we show that, for any $0<\delta\leq 1/4$, if
$z_{1}, z_{2}\in B(z_{f}, \delta)$ for all $t\geq t_{\delta}$ for
some $t_{\delta}>0$, then there exists some $t_{\delta
/2}>t_{\delta}$ such that $z_{1}, z_{2}\in B(z_{f},
\frac{\delta}{2})$ for all $t\geq t_{\delta/2}$. As a consequence,
it is readily seen that the part of (II) concerning $z_1$ and $z_2$
is true.

Let us prove the result for $z_{1}$. The same argument applies to
$z_{2}$. Assume that $z_{1}, z_{2}\in B(z_{f},\delta)$ for all
$t\geq t_{\delta}$ for some $t_{\delta}>0$. First we recall that,
from Proposition \ref{Prop:sigma}, if $z_{2}\in
B(z_{f}, \delta)$, then $|\sigma^{[k]}(z_{2})-z_{f}|<\lambda^{k}_{1/4}%
\delta<\frac{\delta}{3}$, which further implies that
\begin{equation}\label{z-f}
|f(\sigma^{[k]}(z_{2}))-z_{f}|=|f(\sigma^{[k]}(z_{2}))-f(z_{f})|<\lambda^{k}_{1/4}
\delta<\frac{\delta}{3} %
\end{equation}
(see Theorem \ref{Th:simulation}). Let us denote $z_{1}=(y_{1},
y_{2}, y_{3})$, $z_{2}=(v_{1}, v_{2}, v_{3})$, and $f=(f_{1},
f_{2}, f_{3})$, where $y_{i}, v_{i}\in\mathbb{R}$ and $f_{i}:\mathbb{R}^{3}%
\to\mathbb{R}$, $i=1, 2, 3$. Without loss of generality we prove the
result component-wise for the first component, that is, we show that
there is a $t_{\delta/2}>t_{\delta}$ such that
$|y_{1}(t)|<\frac{\delta}{2}$ for all $t>t_{\delta/2}$ (recall that
$z_{f}=(0, 0, m)$ and thus
$|f_{1}(\sigma^{[k]}(z_{2}))|<\frac{\delta}{3}$ for all
$t>t_{\delta}$ by (\ref{z-f}) and the assumption). There are two
cases to be considered. Case 1: If there exists a time
$\tilde{t}>t_{\delta}$ such that $y_{1}(\tilde{t})\in B(0,
\delta/2)$, then we only need to show that $y_{1}(t)\in B(0,
\delta/2)$ for all $t>\tilde{t}$. For any $t>\tilde{t}$, if
$y_{1}=y_{1}(t)>\frac{\delta}{3}$, then
$f_{1}(\sigma^{[k]}(z_{2}))-y_{1}<0$ and so
$y^{\prime}_{1}=c_{1}(f_{1}(\sigma^{[k]}(z_{2}))-y_{1})^{3}
\zeta_{\epsilon_{1}}<0$ for $\zeta_{\epsilon_{1}}>\frac{9}{10}$.
Thus $y_{1}$ will be decreasing until it reaches the value
$f_{1}(\sigma ^{[k]}(z_{2}))$. In other words, for any
$t>\tilde{t}$, $y_{1}$ cannot surpass $\frac{\delta}{2}$. Similarly,
if $y_{1}=y_{1}(t)<-\frac{\delta }{3}$, then
$f_{1}(\sigma^{[k]}(z_{2}))-y_{1}>0$, which further implies that
$y^{\prime}_{1}=c_{1}(f_{1}(\sigma^{[k]}(z_{2}))-y_{1})^{3}\zeta
_{\epsilon_{1}}>0$. Therefore $y_{1}$ will be increasing until it
reaches $f_{1}(\sigma^{[k]}(z_{2}))$, which implies that $y_{1}$
cannot decrease below $-\frac{\delta}{2}$. We have now proved that
$y_{1}(t)\in B(0, \frac{\delta }{2})$ for all $t>t_{\delta/2}$ with
$t_{\delta/2}=\tilde{t}$. Case 2: If for all $t>t_{\delta}$,
$y_{1}\not \in B(0, \frac{\delta}{2})$. Without loss of generality,
let us assume that there is some $t>t_{\delta}$ such that
$y_{1}(t)>\frac{\delta}{2}$. Then, since $|f_{1}(\sigma^{[k]}(z_{2}%
))|<\frac{\delta}{3}$ for all $t>t_{\delta}$, we have $y^{\prime}_{1}%
=c_{1}(f_{1}(\sigma^{[k]}(z_{2}))-y_{1})^{3}\zeta_{\epsilon_{1}}<0$.
It follows that $y_{1}$ will be decreasing, passing the value
$\frac{\delta}{2}$, until it reaches $f_{1}(\sigma^{[k]}(z_{2}))$.
In other words, there will be a time $\tilde{t}$ such that
$y_{1}(\tilde{t})\in B(0, \frac{\delta}{2})$. Then it follows from
the first case that for any $t>\tilde{t}$, $y_{1}(t)\in B(0,
\frac{\delta}{2})$. This is a contradiction. Therefore, the second
case is invalid.

It remains to show that if there exists some $t_{0}>0$ such that
$(z_{1}(t_0), z_{2}(t_0), u(t_0), \tau(t_0))\in U_{f}$, then $\tau
(t)\in (-1, 5)$ for all $t>t_0$ and $\tau(t)\rightarrow 0$ as
$t\rightarrow\infty$. We recall that $\tau^{\prime}=2(u+\xi_{2}(v_{3}-(m-1),\tau+1))^{\alpha+1}\xi_{2}(m-v_{3}%
,10)-\xi_{2}(v_{3}-(m-1),10+10\tau^{2})\tau$. On the other hand, by
the assumption and the proof in the previous paragraph, we have
$v_{3}\rightarrow m$ as $t\rightarrow\infty$. Thus, as $t\to\infty
$, $\xi_{2}(m-v_{3}, 10)\to\xi_{2}(0, 10)=0$ and $|\xi_{2}(v_{3}%
-(m-1),10+10\tau^{2})-1|<\frac{1}{10+10\tau^{2}}$ for sufficiently
large $t$ satisfying $|(v_{3}(t)-(m-1))-1|<\frac{1}{4}$ (see
Proposition \ref{Prop:l2}), or in other words, for sufficiently
large $t$
\[
1-\frac{1}{10}\leq1-\frac{1}{10+10\tau^{2}}<\xi_{2}(v_{3}-(m-1),10+10\tau
^{2})<1+\frac{1}{10+10\tau^{2}}\leq1+\frac{1}{10}
\]
Therefore, for sufficiently large $t$,%
\[
\mbox{$\tau^{\prime}\simeq -a\tau$ for some positive number $a$}
\]
which implies that the part of (II) concerning $\tau$ is true. The
proof for (II) is then complete. (We note that the time $t_0$ in the
assumption $(z_{1}(t_0), z_{2}(t_0), u(t_0), \tau(t_0))\in U_{f}$
may be different in the previous proofs for the three parts of (II)
concerning $u$, $z_1$ and $z_2$, and $\tau$. It suffices to pick the
maximum $t_0$ among the three.)

We have thus shown that the equilibrium point $(z_{f},z_{f},0,0)$ is
indeed a sink of system (\ref{Eq:New_TM2}).  In subsection
\ref{Sec:Hyperbolic} we will show that $(z_{f},z_{f},0,0)$ is a
hyperbolic sink.

\subsection{The system (\ref{Eq:New_TM2}) still simulates the
Turing machine $M$\label{Sec:Simulate_TM}}

In this subsection we show that system (\ref{Eq:New_TM2}) with the
(fixed) initial conditions $u(0)=1$ and $\tau(0)=4$ will still
simulate the machine $M$ before $M$ halts, including the case that
$M$ never halts. We only need to study the system for $t\geq 0$. %

First we show that if $M$ hasn't yet halted, then
\begin{equation}
1<\tau^{\prime}(t)\leq5e^{\frac{5}{4}t} \label{Eq:u_2}%
\end{equation}
for all $t>0$. If $M$ hasn't yet halted, then it follows from
(\ref{Eq:Def_alpha}) that $\alpha\in\left( -\frac{5}{4},-1\text{
}\right]  $. Moreover, since $M$ has not halted, $m-v_{3}\geq3/4$,
which leads to $1-(m-v_3)\leq 1/4$, and thus
$\xi_{2}(m-v_{3},10)\in\lbrack9/10,1)$ by Proposition \ref{Prop:l2}.
Knowing that $\xi_{2}(v_{3}%
-(m-1),\tau+1)>0$, $0<\xi_{2}(x, y)<1$ for all $x\in\mathbb{R}$ and
$y>0$ (by definition of $\xi_{2}$), $u(t)=e^{-t}$ (since
$u^{\prime}=-u$ and $u(0)=1$) and $\tau>-1$, it follows from
(\ref{Eq:u_2_derivative}) that %
\begin{align*}
\tau^{\prime}  &  \leq\frac{2(u+\xi_{2}(v_{3}%
-(m-1),1+\tau))\xi_{2}(m-v_{3},10)}{(u+\xi_{2}(v_{3}%
-(m-1),1+\tau))^{-\alpha}} + \xi_{2}(v_{3}-(m-1),10+10\tau^{2})\\
&
<\frac{2(1+1)}{u^{-\alpha}}+1\leq\frac{2}{u^{\frac{5}{4}}}+1=4e^{\frac{5}{4}t}+1\leq 5e^{\frac{5}{4}t}%
\end{align*}

Next we show that $\tau^{\prime}(t)>1$ for $t>0$. It can be verified
easily that under the assumption that $\tau(0)=4$, $\tau(t)>0$ for
all $t>0$.   Then from (\ref{xi_two_bound}) it follows that
$1/\tau\geq
1/(\tau+1)\geq\xi_{2}(v_{3}-(m-1),\tau+1)>0$ and $1/\tau^{2}\geq\xi_{2}%
(v_{3}-(m-1),10+10\tau^{2})>0$. Using (\ref{Eq:u_2_derivative}) again, we get%
\begin{align*}
\tau^{\prime}  &
=\frac{2\xi_{2}(m-v_{3},10)}{(u+\xi_{2}(v_{3}-(m-1),1+\tau
))^{-\alpha-1}}-\xi_{2}(v_{3}-(m-1),10+10\tau^{2})\tau\\
&  \geq\frac{2\frac{9}{10}}{\left(  u+\frac{1}{\tau}\right)  ^{-\alpha-1}}%
-\frac{1}{\tau^{2}}\tau\geq\frac{9/5}{\left( u+\frac{1}{\tau}\right)
^{-\alpha-1}}-\frac{1}{\tau}\geq\frac{9/5}{\left(
1+\frac{1}{\tau}\right)
^{1/4}}-\frac{1}{\tau}\geq\frac{9/5}{1+\frac{1}{\tau}}-\frac{1}{\tau}%
\end{align*}
We note that the last expression in the above estimate is an
increasing function of $\tau$. Since $\tau(0)=4$, one has
\[
\tau^{\prime}(0)>\frac{9/5}{1+\frac{1}{4}}-\frac{1}{4}=\frac{36}{25}%
-\frac{1}{4}>1
\]
We conclude that $\tau^{\prime}(t)>1$ for all $t>0$.  From
(\ref{Eq:u_2}) we know that,  before $M$ halts, $\tau$
grows with $\tau^{\prime}(t)>1$, which implies that %
$\tau\rightarrow+\infty$ as $t\rightarrow+\infty$ if $M$ never
halts.

The next question  is whether the simulation of $M$ is still being
faithfully carried out by the modified system (\ref{Eq:New_TM2})
until $M$ halts. We recall that the machine $M$ is simulated by
system (\ref{3:Branicky_analytic}) in
the sense that $|z_{1}(j)-f^{[j]}(x_{0})|\leq\eta$ and $|z_{2}(j)-f^{[j]}%
(x_{0})|\leq\eta$ for some $0<\eta<1/4$ and all $j\in\mathbb{N}$
(see (\ref{3:z1_after_time_1/2}) and (\ref{Eq:eta})), where
$x_{0}\in\mathbb{N}^{3}$ codes the initial configuration. Thus the
time needed to complete a cycle (i.e. from one configuration to the
next) is one unit. Now with $\tau$ replacing $t$ as input to the
\textquotedblleft clocking functions\textquotedblright\
$\zeta_{\varepsilon_{1}}$ and $\zeta _{\varepsilon_{2}}$ in the
modified system (\ref{Eq:New_TM2}), we face a new
problem: Since $\tau$ grows at a faster rate than $t$, $1<\tau^{\prime}%
(t)\leq5e^{\frac{5}{4}t}$, the time needed to complete a cycle
becomes shorter. So we need to analyze the effect of this speeded-up
phenomenon.

As we have seen in Subsections \ref{Sec:Branickyinfty} and
\ref{Sec:Branickyanalytic}, in each iteration \textquotedblleft
cycle,\textquotedblright\thinspace\ $[j,j+1]$, of system
(\ref{3:Branicky_analytic}), one of the two variables, $z_{1}$,
$z_{2}$, is \textquotedblleft dormant\textquotedblright \ while the
other is active and is being updated within a $\frac{1}{4}$-vicinity
of the target $f^{[j]}(x_{0})$ during the time period
$[j,j+\frac{1}{2}]$; then, during the later half
$[j+\frac{1}{2},j+1]$ of the cycle, the dormant variable becomes
active and is being updated within the $\frac{1}{4}$-vicinity of the
same target while the active one stays put. The only problem that
may arise with a dormant variable is when the cycle is too long and
thus the dormant may accumulate too much error. However, since we
now have shorter cycles, this problem won't occur. The problem here
is of another type. With system (\ref{Eq:New_TM2}), the iteration
cycles no longer have the uniform length $(j+1)-j=1$, $j\in
\mathbb{N}$, as in (\ref{3:Branicky_analytic}), but become shorter
and shorter as $t\rightarrow\infty$ (possibly exponentially shorter
in the worst-case scenario) because $\tau$ moves faster than $t$
does as (\ref{Eq:u_2}) shows. Thus there may not be enough time to
update the active variable within a $\frac{1}{4}$-vicinity of the
target. Let us look at this in more detail. We note that, in system
(\ref{3:Branicky_analytic}), the reason that $z_{1}$ and $z_{2}$ can
be updated successively on all intervals $[j,j+1]$ (until $M$ halts)
is because both functions $\zeta_{\epsilon_{1}}$ and
$\zeta_{\epsilon_{2}}$, as functions of $t$, have period $1$; thus
it is possible to pick a constant $c_{1}$ such that, for all
$j\in\mathbb{N}$, the target-estimate
\[
\frac{1}{16}\geq\frac{1}{2c_{1}\int_{0}^{1/2}\zeta_{\epsilon_{1}}%
(t)dt}>(f^{[j]}(x_{0})-z_{1}(j+1/2))^{2}%
\]
holds (see (\ref{Eq:target})), which in turn implies that $|f^{[j]}%
(x_{0})-z_{1}(j+1/2)|\leq\frac{1}{4}$ for all $j$.
(Similarly one can select a constant $c_{2}$ satisfying the
target-estimate
for $z_{2}$.) Now with the speeded-up system, the half-cycles $[T_{j}%
,T_{j+1}]$ may decrease without a lower bound, and thus it becomes
impossible to pick a constant $c$ such that the left-hand side of
the target-estimate
\begin{equation}
\frac{1}{16}\geq\frac{1}{2c\int_{T_{j}}^{T_{j+1}}\zeta_{\epsilon_{1}}%
(\tau(t))dt}>(f^{[j]}(x_{0})-z_{1}(T_{j+1}))^{2}\label{target_speed}%
\end{equation}
holds for all $j$. We note that the right-hand side of the above
inequality is always true (see the derivation of (\ref{Eq:target})).
Of course, we cannot select a constant $c_{j}$ for each $j$.
Instead, we solve the problem by multiplying $\phi(t)$ by a function
$\chi(t)$ so that
\begin{equation}
\int_{t_{0}}^{t_{1}}\phi(t)\chi(t)dt\geq\int_{0}^{1/2}\phi
(t)dt\label{Eq:newterm}%
\end{equation}
where $\phi(t)$ corresponds to $\zeta_{\epsilon_{1}}$ or $\zeta_{\epsilon_{2}%
}$ in system (\ref{Eq:New_TM2}) and $[t_{0},t_{1}]$ is an arbitrary
half-cycle, $0<t_{1}-t_{0}<\frac{1}{2}$. The underlying idea here is
that, to compensate for the loss in time, we increase the magnitude
of the function
$\phi(t)$ to $\phi(t)\chi(t)$ so that the integrals $\int_{t_{0}}^{t_{1}}%
\phi(t)\chi(t)dt$ are uniformly bounded below by
$\int_{0}^{1/2}\phi(t)dt$ for all cycles; thus the target-estimate
(\ref{target_speed}) would hold for all cycles with the constant $c$
being $c_{1}$ or $c_{2}$ as in system (\ref{3:Branicky_analytic}).
Or more intuitively, one may notice that the function $\phi(t)$ is a
function with periodic pulses which switch between $\simeq0$ and
$\simeq1$. The value of $\int_{j}^{j+1/2}\phi(t)dt$,
$j\in\mathbb{N},$ is the area under the $j$th active pulse of
$\phi(t)$. The problem with $\phi(\tau(t))$ is that the durations of
the pulses get shorter as $t$ increases. But we would like to
continue to use $\int_{0}^{1/2}\phi(t)$ as a lower bound for the
area under each active pulse in the speeded-up system. In order to
achieve this, we simply use the pulses with increasing heights; that
is, if the duration of each pulse decreases by at most a factor of
$\chi(t)$, then we increase the height of that pulse accordingly and
thus
maintain the area under each pulse at the level of at least $\int_{0}%
^{1/2}\phi(t)dt$.

Now the details. To obtain $\chi(t)$, we need to know how small
$t_{1}-t_{0}$
can be. Recall from (\ref{Eq:u_2}) that $\tau^{\prime}(t)\leq5e^{\frac{5}{4}%
t}$. Since $\tau(t_{1})-\tau(t_{0})=1/2$ ($\zeta_{\epsilon_{2}}$
switches from $\simeq1$ to $\simeq0$ and vice-versa when its
argument increments by $1/2$),
one has%
\begin{gather*}
\frac{1}{2}=\tau(t_{1})-\tau(t_{0})=\int_{t_{0}}^{t_{1}}\tau^{\prime}%
(t)\leq5\int_{t_{0}}^{t_{1}}e^{\frac{5}{4}t}dt=4\left( e^{\frac
{5}{4}t_{1}}-e^{\frac{5}{4}t_{0}}\right)  \text{ \ \ }\Rightarrow\\
4e^{\frac{5}{4}t_{0}}\left(  e^{\frac{5}{4}\Delta t}-1\right)
\geq\frac{1}{2}\text{ \ \ }\Rightarrow\text{ \ \ }\Delta t\geq\frac{4}{5}%
\ln\left(  1+\frac{1}{8}e^{-\frac{5}{4}t_{0}}\right)
\end{gather*}
where $\Delta t=t_{1}-t_{0}$. It can be proved that
\[
a(t)=\frac{4}{5}\ln\left(  1+\frac{1}{8}e^{-\frac{5}{4}t}\right)
>\frac{e^{-\frac{5}{4}t}}{15}=b(t)
\]
for $t\geq0$ (by showing that $a(0)-b(0)>0$, $\lim_{t\rightarrow\infty}%
\frac{a(t)}{b(t)}=\frac{3}{2}>1$, and $a(t)-b(t)$ has a unique
critical point on $(0,\infty)$ at $t=\frac{4\ln 2}{5}$, which gives
the maximum $\frac{4}{5}\ln\frac{17}{16}-\frac{1}{30}$). We omit
these straightforward
calculations. Hence%
\[
\Delta t>\frac{e^{-\frac{5}{4}t}}{15}%
\]
Therefore it is sufficient to multiply $\phi$ (i.e.$\
\zeta_{\epsilon_{1}}$ and $\zeta_{\epsilon_{2}}$) by
$\chi(t)=15e^{\frac{5}{4}t}$ to obtain (\ref{Eq:newterm}). However,
since we desire an autonomous system and $t$ is replaced by $\tau$
in (\ref{Eq:New_TM2}), we need to express $\chi(t)$ in terms of
$\tau$. This can easily be  done by replacing $t$ with $\tau$: from
(\ref{Eq:u_2}) and $\tau(0)=4$ (see the beginning of this section), one easily concludes that one can take
\begin{equation}
\chi(\tau)=15e^{2\tau}\label{chi_tau}%
\end{equation}
since
$\chi(\tau)\geq15e^{\frac{5}{4}t}$.
Because both $\zeta_{\epsilon_{1}}$ and $\zeta_{\epsilon_{2}}$ are
multiplied by $\chi$, the error present in $z_{1}^{\prime}$ and
$z_{2}^{\prime}$ when these variables are dormant is also multiplied
by the same factor. To cancel this effect, we need to get
$\zeta_{\epsilon_{1}}$ and $\zeta_{\epsilon_{2}}$ closer to 0 in the
same proportion. This is done by using $\zeta
_{\frac{\epsilon_{1}}{\chi(\tau)}}$ and
$\zeta_{\frac{\epsilon_{2}}{\chi (\tau)}}$ in (\ref{Eq:New_TM2})
instead of $\zeta_{\epsilon_{1}}$ and $\zeta_{\epsilon_{2}}$. After
replacing $\zeta_{\epsilon_{1}}$ and $\zeta_{\epsilon_{2}}$ by
$\chi\zeta_{\frac{\epsilon_{1}}{\chi(\tau)}}$ and
$\chi\zeta_{\frac{\epsilon_{2}}{\chi(\tau)}}$ in system (\ref{Eq:New_TM2}%
), it is not difficult to see that the system now simulates the
machine $M$, using Construction \ref{Construction_Branicky}. We also
note that $(z_{f},z_{f},0,0)$ remains a sink of system
(\ref{Eq:New_TM2}). For the sake of readability, we will denote
$\zeta_{\frac{\epsilon _{1}}{\chi(t)}}$ and
$\zeta_{\frac{\epsilon_{2}}{\chi(t)}}$ simply as
$\zeta_{\epsilon_{1}}$ and $\zeta_{\epsilon_{2}}$, respectively.

\subsection{The sink is hyperbolic\label{Sec:Hyperbolic}}

Now that we have shown that $z_{equilibrium}$ is a sink of system
(\ref{Eq:New_TM2}), it remains to prove that $z_{equilibrium}$ is
hyperbolic. It suffices to show that the Jacobian of
(\ref{Eq:New_TM2}) at $z_{equilibrium}$ only admits eigenvalues with
negative real parts.

However there is yet another problem with our system
(\ref{Eq:New_TM2}). We note that the first two equations in system
(\ref{Eq:New_TM2}) rely on a certain type of targeting equations
(see (\ref{3:ODE_target})), which in essence can be reduced to
\begin{equation}
z^{\prime}=-z^{3}\label{Eq:cubic}%
\end{equation}
and for which the equilibrium point 0 is not hyperbolic. Therefore,
instead of choosing an equation of type (\ref{3:ODE_target}) or,
more generally, of
 type (\ref{3:Perturbed_target}), we choose an equation with the format%
\[
z^{\prime}=\left(
c(\overline{b}(t)-z)^{3}+(\overline{b}(t)-z)\right) \phi(t)
\]
Since $\overline{b}(t)-z$ always has the same sign as $c(\overline
{b}(t)-z)^{3}$, it is not difficult to see that adding the term
$(\overline {b}(t)-z)$ will not alter the constructions of the
previous sections. Thus the following (final) system
\begin{equation}
\left\{
\begin{array}
[c]{l}%
z_{1}^{\prime}=(c_{1}(f\circ\sigma^{\lbrack k]}(z_{2})-z_{1})^{3}%
+(f\circ\sigma^{\lbrack k]}(z_{2})-z_{1}))\chi(\tau)\zeta_{\epsilon_{1}}%
(\tau)\\
z_{2}^{\prime}=(c_{2}(\sigma^{\lbrack
n]}(z_{1})-z_{2})^{3}+(\sigma^{\lbrack
n]}(z_{1})-z_{2}))\,\chi(\tau)\zeta_{\epsilon_{2}}(-\tau)\\
u^{\prime}=-u\\
\tau^{\prime}=2(u+\xi_{2}(v_{3}-(m-1),\tau+1))^{\alpha}\xi_{2}(m-v_{3}%
,10)-\xi_{2}(v_{3}-(m-1),10+10\tau^{2})\tau
\end{array}
\right.  \label{Eq:New_TM3}%
\end{equation}
will still simulate our universal Turing machine $M$ and have
$z_{equilibrium}$ as a sink as the previous system
(\ref{Eq:New_TM2}) does.

Now we show that that $z_{equilibrium}$ is a hyperbolic sink of
system (\ref{Eq:New_TM3}) by computing the eigenvalues of the
Jacobian of (\ref{Eq:New_TM3}) at $z_{equilibrium}$.
Note that in
(\ref{Eq:New_TM3}) the components $z_{1}$ and $z_{2}$ actually belong to
$\mathbb{R}^{3}$. Fully expanding the system (\ref{Eq:New_TM3}), one obtains%
\begin{equation}
\left\{
\begin{array}
[c]{l}%
y_{1}^{\prime}=(c_{1}(h_{1}(v_{1},v_{2},v_{3})-y_{1})^{3}+(h_{1}(v_{1}%
,v_{2},v_{3})-y_{1}))\chi(\tau)\zeta_{\epsilon_{1}}(\tau)\\
v_{1}^{\prime}=(c_{2}(\sigma^{\lbrack n]}(y_{1})-v_{1})+(\sigma^{\lbrack n]}%
(y_{1})-v_{1}))\,\chi(\tau)\zeta_{\epsilon_{2}}(-\tau)\\
y_{2}^{\prime}=(c_{1}(h_{2}(v_{1},v_{2},v_{3})-y_{2})^{3}+(h_{2}(v_{1}%
,v_{2},v_{3})-y_{2}))\chi(\tau)\zeta_{\epsilon_{1}}(\tau)\\
v_{2}^{\prime}=(c_{2}(\sigma^{\lbrack
n]}(y_{2})-v_{2})^{3}+(\sigma^{\lbrack
n]}(y_{2})-v_{2}))\,\chi(\tau)\zeta_{\epsilon_{2}}(-\tau)\\
y_{3}^{\prime}=(c_{1}(h_{3}(v_{1},v_{2},v_{3})-y_{3})^{3}+(h_{3}(v_{1}%
,v_{2},v_{3})-y_{3}))\chi(\tau)\zeta_{\epsilon_{1}}(\tau)\\
v_{3}^{\prime}=(c_{2}(\sigma^{\lbrack
n]}(y_{3})-v_{3})^{3}+(\sigma^{\lbrack
n]}(y_{3})-v_{3}))\,\chi(\tau)\zeta_{\epsilon_{2}}(-\tau)\\
u^{\prime}=-u\\
\tau^{\prime}=2(u+\xi_{2}(v_{3}-(m-1),\tau+1))^{\alpha}\xi_{2}(m-v_{3}%
,10)-\xi_{2}(v_{3}-(m-1),10+10\tau^{2})\tau
\end{array}
\right.  \label{Eq:Full}%
\end{equation}
where $z_{1}=(y_{1},y_{2},y_{3})$, $z_{2}=(v_{1},v_{2},v_{3})$, and
$h=(f_{1}\circ\sigma^{\lbrack k]}, f_{2}\circ\sigma^{\lbrack k]},
f_{3}\circ\sigma^{\lbrack k]})=(h_{1},h_{2},h_{3})$.  Since
$h(z_{f})=f(\sigma^{\lbrack k]}(0), \sigma^{\lbrack k]}(0),
\sigma^{\lbrack k]}(m))=z_{f}$ (see the proof of property (I) in
section  \ref{Sec:Hyper}), it follows that, in (\ref{Eq:Full}), at
$z_{equilibrium}=(z_{f}, z_{f}, 0, 0)$,
\begin{align*}
h_{i}(v_{1},v_{2},v_{3}) &  =y_{i}\\
\sigma^{\lbrack n]}(y_{i}) &  =v_{i}%
\end{align*}
for $i=1,2$, and $3$, which reduces many terms to 0.  Moreover, since
the map $f:\mathbb{R}^{3}\rightarrow\mathbb{R}^{3}$ is contracting near $z_{f}=(0, 0, m)$
(see Theorem \ref{Th:simulation}), so is $h$
(with a contraction factor bounded in absolute value by
$0\leq\lambda<1$), thus
\begin{equation}
\left\Vert Dh(z_{f})(z-z_{f})\right\Vert _{\infty}\leq\lambda\left\Vert
z-z_{f}\right\Vert _{\infty}\label{Eq:aux9}%
\end{equation}
Now pick $z-z_{f}=(1/4,0,0)$. Since
\[
Dh(z_{f})(z-z_{f})=\frac{1}{4}\left(  \frac{\partial h_{1}}{\partial x_{1}%
},\frac{\partial h_{2}}{\partial x_{1}},\frac{\partial h_{3}}{\partial x_{1}%
}\right)
\]
using (\ref{Eq:aux9}) and the sup-norm on both sides, we get%
\begin{align*}
\frac{\lambda}{4} &  =\lambda\left\Vert z-z_{f}\right\Vert _{\infty}%
\geq\left\Vert Dh(z_{f})(z-z_{f})\right\Vert _{\infty}\geq\frac{1}%
{4}\left\vert \frac{\partial h_{i}}{\partial x_{1}}\right\vert \\
&  \Rightarrow\text{ \ \ }\lambda\geq\left\vert \frac{\partial h_{i}}{\partial
x_{1}}\right\vert
\end{align*}
for $i=1,2$, and $3$. Picking $z-z_{f}=(0,1/4,0)$ and
$z-z_{f}=(0,0,1/4)$, and proceeding similarly, we reach the
conclusion that all partial derivatives of $h$ are bounded in
absolute value by $\lambda$ at the point $z_{h}$.
Now we define
$\sigma_{i}=$ $\frac{\partial\sigma^{\lbrack n]}(y_{i})}{\partial y_{i}}$.
Notice that for $n\geq1$,
\begin{equation}
\left\vert \sigma_{i}\right\vert \leq\lambda_{1/4}=0.4\pi-1\approx
0.256637\label{Eq:Imp2}%
\end{equation}
from Proposition \ref{Prop:sigma} and the proof of Lemma
\ref{Lemma:hyperbolic}. Therefore the Jacobian matrix $A$ of (\ref{Eq:Full})
at the point $z_{e}=z_{equilibrium}$ is%
\begin{align*}
&  \left[
\begin{array}
[c]{cccc}%
-\left.  \chi(\tau)\zeta_{\epsilon_{1}}(\tau)\right\vert _{z=z_{e}} & \left.
\frac{\partial h_{1}}{\partial v_{1}}\chi(\tau)\zeta_{\epsilon_{1}}%
(\tau)\right\vert _{z=z_{e}} & 0 & \left.  \frac{\partial h_{1}}{\partial
v_{2}}\chi(\tau)\zeta_{\epsilon_{1}}(\tau)\right\vert _{z=z_{e}}\\
\sigma_{1}\left.  \chi(\tau)\zeta_{\epsilon_{2}}(-\tau)\right\vert _{z=z_{e}}
& -\left.  \chi(\tau)\zeta_{\epsilon_{2}}(-\tau)\right\vert _{z=z_{e}} & 0 &
0\\
0 & \left.  \frac{\partial h_{2}}{\partial v_{1}}\chi(\tau)\zeta_{\epsilon
_{1}}(\tau)\right\vert _{z=z_{e}} & -\left.  \chi(\tau)\zeta_{\epsilon_{1}%
}(\tau)\right\vert _{z=z_{e}} & \left.  \frac{\partial h_{2}}{\partial v_{2}%
}\chi(\tau)\zeta_{\epsilon_{1}}(\tau)\right\vert _{z=z_{e}}\\
0 & 0 & \sigma_{2}\left.  \chi(\tau)\zeta_{\epsilon_{2}}(-\tau)\right\vert
_{z=z_{e}} & -\left.  \chi(\tau)\zeta_{\epsilon_{2}}(-\tau)\right\vert
_{z=z_{e}}\\
0 & \left.  \frac{\partial h_{3}}{\partial v_{1}}\chi(\tau)\zeta_{\epsilon
_{1}}(\tau)\right\vert _{z=z_{e}} & 0 & \left.  \frac{\partial h_{3}}{\partial
v_{2}}\chi(\tau)\zeta_{\epsilon_{1}}(\tau)\right\vert _{z=z_{e}}\\
0 & 0 & 0 & 0\\
0 & 0 & 0 & 0\\
0 & 0 & 0 & 0
\end{array}
\right.  \\
&  \left.
\begin{array}
[c]{cccc}%
0 & \left.  \frac{\partial h_{1}}{\partial v_{3}}\chi(\tau)\zeta_{\epsilon
_{1}}(\tau)\right\vert _{z=z_{e}} & 0 & 0\\
0 & 0 & 0 & 0\\
0 & \left.  \frac{\partial h_{2}}{\partial v_{3}}\chi(\tau)\zeta_{\epsilon
_{1}}(\tau)\right\vert _{z=z_{e}} & 0 & 0\\
0 & 0 & 0 & 0\\
-\left.  \chi(\tau)\zeta_{\epsilon_{1}}(\tau)\right\vert _{z=z_{e}} & \left.
\frac{\partial h_{3}}{\partial v_{3}}\chi(\tau)\zeta_{\epsilon_{1}}%
(\tau)\right\vert _{z=z_{e}} & 0 & 0\\
\sigma_{3}\left.  \chi(\tau)\zeta_{\epsilon_{2}}(-\tau)\right\vert _{z=z_{e}}
& -\left.  \chi(\tau)\zeta_{\epsilon_{2}}(-\tau)\right\vert _{z=z_{e}} & 0 &
0\\
0 & 0 & -1 & 0\\
0 & \beta & 0 & -\xi_{2}\left(  1,10\right)
\end{array}
\right]
\end{align*}
where $\beta\in\mathbb{R}$ is well defined (we do not need the
explicit value of this partial derivative, it suffices to know its
existence). To show that $z_{equilibrium}$ is a hyperbolic sink, we
just need to prove that all eigenvalues of the above matrix have
negative real part. Suppose, otherwise, that $A$ admits an
eigenvalue $\mu$ with nonnegative real part. Let
$x=(x_{1},x_{2},x_{3},x_{4},x_{5},x_{6},x_{7},x_{8})\in\mathbb{C}^{8}$
be an eigenvector of $A$ associated to $\mu$. Since $x$ is an
eigenvector, $x\neq0$. We also note that, from (\ref{chi_tau}),
$\left.\chi(\tau)\right\vert
_{z=z_{e}}=\chi(0)=5$. Then, from the equation $Ax=\mu x$ and (\ref{Eq:zetas}%
), one gets the following equations%
\begin{align}
\left.  \chi(\tau)\zeta_{\epsilon_{1}}(\tau)\right\vert _{z=z_{e}}\left(
-x_{1}+\left.  \frac{\partial h_{1}}{\partial v_{1}}\right\vert _{z=z_{e}%
}x_{2}+\left.  \frac{\partial h_{1}}{\partial v_{2}}\right\vert _{z=z_{e}%
}x_{4}+\left.  \frac{\partial h_{1}}{\partial v_{3}}\right\vert _{z=z_{e}%
}x_{6}\right)   &  =\mu x_{1}\text{ \ \ }\Longrightarrow\text{ \ \ }%
\nonumber\\
\frac{1}{1+\frac{\mu}{\left.  \chi(\tau)\zeta_{\epsilon_{1}}(\tau)\right\vert
_{z=z_{e}}}}\left(  \left.  \frac{\partial h_{1}}{\partial v_{1}}\right\vert
_{z=z_{e}}x_{2}+\left.  \frac{\partial h_{1}}{\partial v_{2}}\right\vert
_{z=z_{e}}x_{4}+\left.  \frac{\partial h_{1}}{\partial v_{3}}\right\vert
_{z=z_{e}}x_{6}\right)   &  =x_{1}\text{ \ \ }\Longrightarrow\text{
\ \ }\nonumber\\
\left\vert \left.  \frac{\partial h_{1}}{\partial v_{1}}\right\vert _{z=z_{e}%
}x_{2}+\left.  \frac{\partial h_{1}}{\partial v_{2}}\right\vert _{z=z_{e}%
}x_{4}+\left.  \frac{\partial h_{1}}{\partial v_{3}}\right\vert _{z=z_{e}%
}x_{6}\right\vert  &  \geq\left\vert x_{1}\right\vert \text{ \ \ }%
\Longrightarrow\text{ \ \ }\nonumber\\
\lambda\left(  \left\vert x_{2}\right\vert +\left\vert x_{4}\right\vert
+\left\vert x_{6}\right\vert \right)   &  \geq\left\vert x_{1}\right\vert
\label{Eq:eigen1}%
\end{align}
We also obtain from $Ax=\mu x$ and (\ref{Eq:Imp2})%
\begin{gather}
\left.  \chi(\tau)\zeta_{\epsilon_{2}}(\tau)\right\vert _{z=z_{e}}\left(
\sigma_{1}x_{1}-x_{2}\right)  =\mu x_{2}\text{ \ \ }\Longrightarrow\text{
\ \ }\nonumber\\
\sigma_{1}x_{1}=\left(  1+\frac{\mu}{\left.  \chi(\tau)\zeta_{\epsilon_{2}%
}(-\tau)\right\vert _{z=z_{e}}}\right)  x_{2}\text{ \ \ }\Longrightarrow\text{
\ \ }\nonumber\\
\lambda_{1/4}\left\vert x_{1}\right\vert \geq\left\vert x_{2}\right\vert
\label{Eq:eigen3}%
\end{gather}
Similarly the following can be derived:%
\begin{equation}
\left\{
\begin{array}
[c]{c}%
\lambda\left(  \left\vert x_{2}\right\vert +\left\vert x_{4}\right\vert
+\left\vert x_{6}\right\vert \right)  \geq\left\vert x_{3}\right\vert \\
\lambda\left(  \left\vert x_{2}\right\vert +\left\vert x_{4}\right\vert
+\left\vert x_{6}\right\vert \right)  \geq\left\vert x_{5}\right\vert
\end{array}
\right.  \text{ \ \ and \ \ }\left\{
\begin{array}
[c]{c}%
\lambda_{1/4}\left\vert x_{3}\right\vert \geq\left\vert x_{4}\right\vert \\
\lambda_{1/4}\left\vert x_{5}\right\vert \geq\left\vert x_{6}\right\vert
\end{array}
\right.  \label{Eq:eigen2}%
\end{equation}
Then it follows from (\ref{Eq:eigen1}), (\ref{Eq:eigen3}), and
(\ref{Eq:eigen2}) that
\[
\left\{
\begin{array}
[c]{c}%
\lambda\lambda_{1/4}(\left\vert x_{1}\right\vert +\left\vert x_{3}\right\vert
+\left\vert x_{5}\right\vert )\geq\lambda\left(  \left\vert x_{2}\right\vert
+\left\vert x_{4}\right\vert +\left\vert x_{6}\right\vert \right)
\geq\left\vert x_{1}\right\vert \\
\lambda\lambda_{1/4}(\left\vert x_{1}\right\vert +\left\vert x_{3}\right\vert
+\left\vert x_{5}\right\vert )\geq\lambda\left(  \left\vert x_{2}\right\vert
+\left\vert x_{4}\right\vert +\left\vert x_{6}\right\vert \right)
\geq\left\vert x_{3}\right\vert \\
\lambda\lambda_{1/4}(\left\vert x_{1}\right\vert +\left\vert x_{3}\right\vert
+\left\vert x_{5}\right\vert )\geq\lambda\left(  \left\vert x_{2}\right\vert
+\left\vert x_{4}\right\vert +\left\vert x_{6}\right\vert \right)
\geq\left\vert x_{5}\right\vert
\end{array}
\right.
\]
Adding the inequalities, we get
\[
3\lambda\lambda_{1/4}(\left\vert x_{1}\right\vert +\left\vert x_{3}\right\vert
+\left\vert x_{5}\right\vert )\geq\left\vert x_{1}\right\vert +\left\vert
x_{3}\right\vert +\left\vert x_{5}\right\vert
\]
which implies (note that $0<3\lambda_{1/4}<1$)%
\[
\lambda(\left\vert x_{1}\right\vert +\left\vert x_{3}\right\vert +\left\vert
x_{5}\right\vert )\geq\left\vert x_{1}\right\vert +\left\vert x_{3}\right\vert
+\left\vert x_{5}\right\vert
\]
Since $\left\vert \lambda\right\vert <1$, the above holds true only if
$x_{1}=x_{2}=x_{3}=0$. Moreover, from (\ref{Eq:eigen3}) and (\ref{Eq:eigen2})
one also concludes $x_{2}=x_{4}=x_{6}=0$. Thus the equation $Ax=\mu x$ is
reduced to
\[
\left\{
\begin{array}
[c]{c}%
-x_{7}=\mu x_{7}\\
-\xi_{2}\left(  1,10\right)  x_{8}=\mu x_{8}%
\end{array}
\right.
\]
Since $0<\xi_{2}\left(  1,10\right)  <1$ and $\mu$ has nonnegative
real part, the system above is satisfied only if $x_{7}=x_{8}=0$. In
other words, any eigenvector associated to the eigenvalue $\mu$ is a
zero vector, which is clearly a contradiction. Therefore, all
eigenvalues of $A$ have negative real part, i.e.\ $z_{equilibrium}$
is a hyperbolic sink for the ODE (\ref{Eq:New_TM3}).

We note that all functions used in system (\ref{Eq:New_TM3}) are
computable (they are defined by composing usual functions of
analysis with some of their analytic continuations, and therefore
are computable; see \cite{PR89}, \cite{BHW08}).

\subsection{Proof of Theorem \ref{Th:Main_continuous}\label{Sec:final}}

In the previous subsections \ref{Sec:Simulate_TM} and
\ref{Sec:Hyperbolic}, we have shown that the universal Turing
machine $M$ can be simulated by the ODE (\ref{Eq:New_TM3}) which has
a hyperbolic sink at
$z_{equilibrium}=(z_{f},z_{f},0,0)\in\mathbb{N}^{8}$. The lemma
below follows from the results in those sections.

\begin{lem}
\label{Lema:aux}Suppose that $x_{0}\in\mathbb{N}^{3}$ codes an
initial configuration of $M$ simulated by system (\ref{Eq:New_TM3}),
and let $z_{0}=(x_{0}, x_{0}, 1, 4)$. Then for every point
$z\in\mathbb{R}^{8}$ satisfying $\left\Vert z-z_{0}\right\Vert
<1/16$, one has:
\end{lem}

\begin{enumerate}
\item If $M$ halts on $x_{0}$, then the trajectory starting at $z$ converges
to the hyperbolic sink $z_{equilibrium}$.

\item If $M$ does not halt on $x_{0}$, then the trajectory starting at $z$
does not converge to $z_{equilibrium}$.
\end{enumerate}

We now proceed to the proof of Theorem \ref{Th:Main_continuous}.
\smallskip

\begin{proof}
[of Theorem \ref{Th:Main_continuous}]   Let $g:
\mathbb{R}^6\times (-1,+\infty)\times
(-1,+\infty)\rightarrow\mathbb{R}^8$ be the function in the
right-hand side of system (\ref{Eq:New_TM3}) and let $s=(z_{f},
z_{f}, 0, 0)$. Then $g$ is analytic. As proved in previous
subsections, $s$ is a hyperbolic sink of system (\ref{Eq:New_TM3}),
where $z_{f}$ corresponds to the unique halting configuration of the
universal Turing machine $M$ simulated by system (\ref{Eq:New_TM3}).
Let us denote the basin of attraction of $s$ as $W_{final}$. It then
follows that, for any $z_{0}=(x_{0}, x_{0}, 1, 4)$ with
$x_0\in\mathbb{N}^3$ being an initial configuration of $M$, $M$
halts on $x_0$ iff $z_{0}\in W_{final}$. Moreover, from Lemma
\ref{Lema:aux}, any trajectory starting at a point inside
$B(z_{0},1/16)$ will either converge to $s$ if $M$ halts on $x_{0}$
or not converge to $s$ if $M$ does not halt on $x_{0}$. In other
words, $B(z_{0},1/16)$ is either inside $W_{final}$ (if $M$ halts on
$x_{0}$) or inside $\mathbb{R}^{8}-W_{final}$ (if $M$ does not halt
on $x_{0}$).

Let $\mathfrak{R}$ denote $\mathbb{R}^6\times (-1,+\infty)\times
(-1,+\infty)$. Now suppose that $W_{final}$ is a computable open
subset of $\mathfrak{R}$. Then the distance function
$d_{\mathfrak{R}\backslash W_{final}}$ is computable. Therefore, for
any $x_{0}\in\mathbb{N}^3$,  we can compute
$d_{\mathfrak{R}\backslash W_{final}}(z_{0})$ with a precision of
$1/40$ which yields some rational $q$, where $z_{0}=(x_{0}, x_{0},
1, 4)$. We note that either $d_{\mathfrak{R}\backslash
W_{final}}(z_{0})=0$ (if $z_{0}\notin W_{final}$) or
$d_{\mathfrak{R}\backslash W_{final}}(z_{0})\geq1/16$ (if $z_{0}\in
W_{final}$). In the first case, $q\leq 1/40$, while in the second
case, $q\geq \frac{1}{16}-\frac{1}{40}=\frac{3}{80}$. The following
algorithm then solves the Halting problem, which is absurd: on
initial configuration $x_{0}$, compute $d_{\mathbb{R}^{8}\backslash
W_{final}}(z_{0})$, $z_{0}=(x_{0}, x_{0}, 1, 4)$, with a precision
of $1/40$ yielding some rational $q$. If $q\leq1/40$ then $M$ does
not halt on $x_{0}$; if $q\geq3/80$, then $M$ halts on $x_{0}$.
Therefore $W_{final}$ cannot be computable.

\end{proof}

\bigskip
\footnotesize
\noindent\textit{Acknowledgments.}
D.\ Gra\c{c}a was partially supported by
\emph{Funda\c{c}\~{a}o para a Ci\^{e}ncia e a Tecnologia} and EU FEDER
POCTI/POCI via SQIG - Instituto de Telecomunica\c{c}\~{o}es through the FCT project PEst-OE/EEI/0008/2013.


\normalsize
\baselineskip=17pt


\bibliographystyle{plain}
\bibliography{ContComp}

\end{document}